\documentclass{proc-l}

\usepackage{amssymb}

\usepackage{graphicx}

\usepackage{url}

\newtheorem{theorem}{Theorem}[section]
\newtheorem{lemma}[theorem]{Lemma}

\newtheorem{cor}[theorem]{Corollary}

\theoremstyle{definition}
\newtheorem{definition}[theorem]{Definition}

\theoremstyle{remark}
\newtheorem{remark}[theorem]{Remark}

\usepackage{xcolor}

\newcommand{\jerry}{}

\newcommand{\im}{\mathrm{Im}}
\newcommand{\Ker}{\mathrm{Ker}}
\newcommand{\Z}{\mathbb{Z}}

\newcommand{\idx}{i}

\newcommand{\I}{{I}}
\newcommand{\J}{{J}}
\newcommand{\K}{{K}}

\newcommand{\basis}{\beta}

\usepackage{tikz-cd}

\newcommand{\IK}{\mathrm{IK}}

\usepackage{cite}

\numberwithin{equation}{section}

\begin{document}

% \title[short text for running head]{full title}
\title[Interval Decomposition over a PID]{
% Infinitely-Indexed Persistence Modules over a Principal Ideal Domain: Decomposition and Diagram Invariance
% Interval Decomposition of Infinitely-Indexed Persistence Modules over a Principal Ideal Domain (and Field Independence in Persistent Homology)
% Interval Decomposition of infinite PID persistence modules, and a criterion for field independence in persistent homology
% Interval decomposition and field independence for PID persistence modules and persistent homology
Interval Decomposition of Infinite Persistence Modules over a Principal Ideal Domain and Field Choice in Persistent Homology
}

%    Only \author and \address are required; other information is
%    optional.  Remove any unused author tags.

   % author one information
% \author[short version for running head]{name for top of paper}
\author{Jiajie Luo}
\address{Knowledge Lab, University of Chicago, Chicago, Illinois, United States of America}
\curraddr{}
\email{jerryluo8@uchicago.edu}
\thanks{}

%    author two information
\author{Gregory Henselman-Petrusek}
\address{Pacific Northwest National Laboratory, Richland, Washington, United States of America}
\curraddr{}
\email{gregory.roek@pnnl.gov}
\thanks{}

%    \subjclass is required.

\date{}

\dedicatory{}

%    "Communicated by" -- provide editor's name; required.
\commby{}

%    Abstract is required.
\begin{abstract}
We study pointwise free and finitely-generated persistence modules over a principal ideal domain, indexed by a (possibly infinite) totally-ordered poset category.
We show that such persistence modules admit interval decompositions if and only if every structure map has free cokernel. 
We also show that, in torsion-free settings, the integer persistent homology module of a filtration of topological spaces admits an interval decomposition if and only if the associated persistence diagram is invariant to the choice of coefficient field.
These results generalize prior work where the indexing category is finite.
\end{abstract}

\maketitle

%    Text of article.

% TODOs
% \begin{itemize}
%     \item title is weird. think about renaming
%     \item do we want to keep ``consistent basis with respect to $f$''?
%     \item talk about $\IK$ business. I think we should keep it as is, but  to discuss...
%     \item is it worth discussing field independence? maybe? \greg{actually i think this is very doable. see the discussion in our other paper about universal cycle bases. i don't think we even need to re-hash that argument here. i think we can just point the reader to that section of the other paper and say that the same argument applies.  this would actually be a really nice addition.}
% \end{itemize}

% \greg{It would probably be a good idea to mention that this new result allows us to extend prior results on persistent homology, namely field-independence, to the infinite-indexed setting.}

\section{\textbf{Introduction}}

Let $R$ be a principal ideal domain (PID). 
A \emph{persistence module} with coefficients in $R$ is a functor $f:I\to R\mathrm{-Mod}$,
where the \emph{indexing category} $I$ is a totally-ordered poset category.
Persistence modules \jerry{\cite{edelsbrunner_topological_2002, ghrist_barcodes_2007, Polterovich2020}} arise from and have been widely studied in persistent homology, in which setting $R$ is typically a field and $f$ maps every index $i \in I$ to a finite-dimensional vector space over $R$. 
{In this setting, any persistence module admits an interval decomposition; that is, it can be decomposed as a direct sum of indecomposables known as ``interval modules''.} 
% Any persistence module that meets this condition admits an interval decomposition; that is, it can be decomposed as a direct sum of indecomposables known as ``interval modules''. 
% Interval decompositions of persistence modules are key to the study of persistent homology, in that they can be used to capture important topological structure within data sets. 
{Interval decompositions are key to persistent homology, in that they capture important topological structure within data.}
If one removes the condition that the coefficient ring $R$ be a field, then interval decompositions are no longer guaranteed to exist \jerry{\cite{oudot_persistence_2015, gabriel_unzerlegbare_1972}}. 

In this paper, we study persistence modules that are pointwise free and finitely-generated, in which the coefficient ring is a PID and the indexing category is any totally-ordered poset category.
We introduce necessary and sufficient conditions for such modules to decompose (see Section \ref{sec:background} for definitions):

\begin{theorem}
    [Main result]
    \label{thm:problem_statement}
    Let $f:I\to R\mathrm{-Mod}$ be a pointwise free and finitely-generated persistence module, where $I$ {may be} infinite and $R$ is a PID. Then $f$ admits an interval decomposition if and only if every  map $f(a\leq b)$ has free cokernel. 
\end{theorem}

The existence of interval decompositions for modules with \emph{field coefficients} has been widely attributed to Gabriel \cite{gabriel_unzerlegbare_1972}, 
whose work showed that indecomposable representations of quivers of type $A_n$ have only finitely many isomorphism classes, which coincide exactly with interval modules.
% This result has since been generalized to persistence modules indexed by any totally-ordered set \cite{crawleybevey_2015}, as well as persistence modules indexed by a small category \cite{botnan_crawleybevey_2020}.
This result has since been generalized to persistence modules indexed by any totally-ordered set \cite{crawleybevey_2015}, as well as a small category \cite{botnan_crawleybevey_2020}.\footnote{\jerry{The results in \cite{crawleybevey_2015,botnan_crawleybevey_2020} had been stated earlier (without proof) by Gabriel and Ro{\u{\i}}ter \cite[Pages 30--32]{GabrielRoiter1992}.}}
{Interval decompositions have also been generalized to zigzag persistence.}
Carlson and de Silva \cite{carlsson_desilva_2010} showed that interval decompositions exist for finitely-indexed zigzag persistence modules; 
Botnan \cite{botnan_2017} later showed the same in the infinitely-indexed setting. 
\jerry{Igusa et al. \cite{IgusaRockTodorov2023} provided a proof for the more general continuous-zigzag setting, independent of \cite{botnan_2017}.}
% Hang and Mio \cite{hang_correspondence_2023} generalized {the decomposition} to a broader class of tame correspondence modules, especially functors from $(\R, \le)$ to the category of finite-dimensional vector spaces and partially-defined linear relations between them.
Hang and Mio \cite{hang_correspondence_2023} generalized these results further, with interval decomposition of tame correspondence modules.

The problem of decomposing modules with \emph{PID coefficients} is motivated by growing mathematical and scientific interest in modules with non-field coefficients, such as in 
% relation to 
Conley index theory \cite{conley_integers, Luo2023}, generalized persistence diagrams \cite{patel_generalized_2018}, projective coordinatization \cite{de_silva_persistent_2009,perea_multiscale_2018}, and learned representations in neural networks \cite{walch_torsion_2025}.

However, this problem presents nontrivial structural challenges which are absent from the  works described above. First, {\emph{the endomorphism rings of indecomposable modules with PID coefficients are nonlocal}}, which complicates the approach of \cite{botnan_crawleybevey_2020}. In particular, Theorem 1.1 and Lemma 2.1 of that work bear heavily on the proof of interval decomposition \cite[Theorem 1.2]{botnan_crawleybevey_2020}, but all three results fail over a PID.  Second, \emph{arbitrary submodules of free $R$-modules need not have complements}, which complicates the approach of \cite{crawleybevey_2015}. For example, the limiting object $V_I^-$ lacks a complement to $V_I^+$ in general, which impedes the construction in \cite[Section 5]{crawleybevey_2015}, and creates additional challenges in \cite[Section 6]{crawleybevey_2015}.

To address these problems, we abstract certain aspects of the algebraic structure, focusing instead on the structure of subobject lattices. 
Specifically, we show that the functorial properties of the \emph{saecular lattice} \cite{Luo2023, ghrist_saecular_2021} {can be used} to build decompositions constructively from images and kernels.

Our work complements and builds on \cite{Luo2023}, which established necessary and sufficient conditions for interval decompositions of persistence modules over PIDs with \emph{finite} indexing categories. 
{The core element of \cite{Luo2023} is a \textbf{polynomial-time algorithm} to compute decompositions explicitly.}
% \cite{Luo2023} proved sufficiency by explicitly computing interval decompositions using a \textbf{polynomial-time algorithm}. 
% \greg{This work proved sufficiency by formulating {\textbf{polynomial time algorithms}} to construct decompositions, explicitly.} 
{While key ideas in \cite{Luo2023} remain relevant, the generality of our indexing category requires nontrivial adaptation.}
% While the key ideas from \cite{Luo2023} remain relevant, the increased generality of our indexing category motivates us to take a different constructive approach, much as \cite{botnan_2017} reached to new methods when extending interval decompositions for field coefficients from the finite to the infinite case.
\jerry{Roughly speaking, \cite{Luo2023} constructs an interval decomposition of $f:\{1,\ldots,m\}\to R\mathrm{-Mod}$ by recursively constructing a basis for each $f_i$ using the basis on $f_{i-1}$.
% While this approach accommodates any finite (totally-ordered) indexing category, we cannot proceed by exhaustion for an arbitrary (totally-ordered) indexing category.
In the case where $I$ is infinite we would like to do something similar, but cannot proceed by exhaustion as in the finite case. 
Instead, we appeal to Zorn's lemma, in a manner similar to \cite{botnan_2017}. 
As in that work, we lean heavily on the fact that, for fixed $i$, the set of distinct submodules of form $\Ker(f(i \le j))$ (respectively, $\im(f(k \le k))$ has finite cardinality.}

\section{\textbf{Background and notation}}
\label{sec:background}

In this section, we present the necessary background for persistence modules and set up the notation of this paper. 
We also provide a brief discussion on finitely-indexed persistence modules, which will be relevant to prove the main result, Theorem \ref{thm:problem_statement}.
See \cite{ghrist_barcodes_2007,Polterovich2020,oudot_persistence_2015} for more on persistence modules and homology. 

\subsection{Persistence modules}

Let $R$ be a ring and $I$ a totally-ordered poset category.

\begin{definition}
    A \emph{persistence module} indexed by $I$ with coefficients in $R$ is a functor $f:I\to R\mathrm{-Mod}$ from $I$ to the category of $R$-modules. 
\end{definition}
In other words, $f$ consists of (1) a family $\{f_i\}_{i\in I}$ of $R$-modules and (2) morphisms (which we call \emph{structure maps}) $f(i\leq j):f_i\to f_j$ whenever $i\leq j$, in which $f(i\leq i)=\mathrm{id}_{f_i}$ for all $i\in I$ and $f(i\leq k)= f(j\leq k)\circ f(i\leq j)$ whenever $i\leq j\leq k$. 

For a persistence module $f:I\to R\mathrm{-Mod}$, we denote $f_a$ as $f$ evaluated at $a$, and $f(a\leq b):f_a\to f_b$ as the structure map between $f_a$ and $f_b$. 

\

\noindent\fbox{\parbox{\dimexpr\linewidth-2\fboxsep-2\fboxrule}{
    For the remainder of this work, unless otherwise stated, $f$ will denote a persistence module $f:I\to R\mathrm{-Mod}$ which is  pointwise free and finitely-generated. 
    That is, for each $i\in I$, we have that $f_i$ is a finitely-generated free $R$-module.

    \vspace{1pt}

    We will assume, without loss of generality, that $I$ has a minimal and maximal element, which we denote $-\infty$ and $\infty$, respectively, and that $f_{-\infty}=f_\infty=0$. 
    This assumption simplifies our analysis without loss of generality.
}}

\

\begin{definition}
    Let $f,g:I\to R\mathrm{-Mod}$ be persistence modules. 
    We say that $\phi:f\to g$ is a \emph{morphism} between persistence modules if $\phi$ is a natural transformation. 
    That is, $\phi$ consists of a family $\{\phi_i\}_{i\in I}$ of maps such that for all $a,b\in I$ where $a\leq b$, we have $\phi_b\circ f(a\leq b) = g(a\leq b)\circ \phi_a$.
    An \emph{isomorphism} between persistence modules is a morphism with an inverse. 
\end{definition}

\begin{definition}
    Let $I$ be a totally-ordered poset. 
    An \emph{interval} is a subset $J\subseteq I$ such that for all $a,b\in J$ and $c\in I$ such that $a\leq c\leq b$, we have $c\in J$. 
\end{definition}

\begin{definition}
    An \emph{interval module} is a persistence module of the form $h^J:I\to R\mathrm{-Mod}$, where $J$ is an interval,
% \[h^J_a = \begin{cases}
%     R & a\in J \\
%     0 & \text{otherwise}\,,
% \end{cases}\]
% and \[h^J(a\leq b) = \begin{cases}
%     id_R & a,b\in R \\
%     0 & \mathrm{otherwise}\,.
% \end{cases}\]
\begin{align*}
h^J_a = \begin{cases}
    R & a\in J \\
    0 & \text{otherwise}\,,
\end{cases}
&&
h^J(a\leq b) = \begin{cases}
    id_R & a,b\in J \\
    0 & \mathrm{otherwise}\,.
\end{cases}
\end{align*}
An \emph{interval submodule} of $f$ is any persistence submodule that is isomorphic to an interval module. 
An \emph{interval decomposition} of $f$ is a direct-sum decomposition in which each summand is an interval submodule.
\end{definition}

% \begin{definition}
    
% % An \emph{interval decomposition} of $f$ is a decomposition 
% % \[
% % f\cong \bigoplus_k h^{J_k}
% % \]
% % into interval modules.
% % Here, each $J_k\subseteq I$ is an interval.
% \jerry{An \emph{interval decomposition} of $f$ is a decomposition into interval modules.}
% % Here, each $J_k\subseteq I$ is an interval.
% \end{definition}

\begin{remark}
    % % Equivalently, an interval decomposition of $f$ can be characterized as a choice of bases $(\beta_i)_{i\in I}$, where each $\beta_i\subseteq f_i$ is a basis, such that for all $a,b\in I$ where $a\leq b$, 
    % \jerry{Equivalently, an interval decomposition of $f$ is a choice of bases $(\beta_i)_{i\in I}$, where $\beta_i\subseteq f_i$, such that for all $a,b\in I$ where $a\leq b$,}
    % \jerry{[FLAG: do we use this definition?!?]}
    % \begin{itemize}
    %     % \item the structure map $f(a\le b)$ sends each element of $\beta_a$ either to zero or to an element of $\beta_b$;
    %     \item $f(a\le b)$ sends each element of $\beta_a$ either to zero or to an element of $\beta_b$;
    %     \item no pair of distinct elements of $\beta_a$ map to the same element in $\beta_b$.
    % \end{itemize}
    $f$ admitting an interval decomposition is equivalent to the condition that, for all $a \le b$, the matrix representation of $f(a \le b)$ with respect to bases $\beta_a$ and $\beta_b$ is a matching matrix\footnote{A \emph{matching matrix} is  a matrix with coefficients in $\{0,1\}$, where each row (respectively, each column) contains at most one nonzero entry}.
\end{remark}

Much of our discussion will focus on images and kernels, so we will use some convenient shorthand for these modules:

\begin{definition}\label{def:imker}
	Fix $a\in I$. For $x,y\in I$, 
	we define \jerry{the following submodules of the module $f_a$:} 
 %    $$\mathrm{Ker}[a,y]=\begin{cases}
	% \ker(f(a\leq y))\,, & a\leq y \\ 
	% 0\,, & \text{otherwise}
	% \end{cases}$$ and $$\mathrm{Im}[x,a]=\begin{cases}
	% \im(f(x\leq a))\,, & x\leq a \\
	% f_a\,, & \text{otherwise}\,.
	% \end{cases}$$
    \begin{align*}
\mathrm{Ker}[a,y]=\begin{cases}
    \ker(f(a\leq y))\,, & a\leq y \\ 
    0\,, & \text{otherwise}\,,
\end{cases}
&&
\mathrm{Im}[x,a]=\begin{cases}
    \im(f(x\leq a))\,, & x\leq a \\
    f_a\,, & \text{otherwise}\,.
\end{cases} 
\end{align*}
    We also define
    $$
    \IK_a[x,y]=\im[x,a]\cap\Ker[a,y].
    $$
% \textcolor{red}{Concretely, $\Ker[a,y], \im[x,a]$, and $\IK_a[x,y]$ are all submodules of the module $f_a$.}
\end{definition}

\begin{remark}
    For $y_1,y_2\in I$ such that $y_1\leq y_2$, we have $\Ker[a,y_1]\subseteq \Ker[a,y_2]$. 
    Similarly, for $x_1,x_2\in I$ such that $x_1\leq x_2$, we have $\im[x_1,a]\subseteq \im[x_2,a]$.
\end{remark}

\begin{definition}
    Let $M$ be a $R$-module, and let $A,B\subseteq M$ be submodules such that $A\subseteq B$. 
    We say that $A$ has a \emph{complement} in $B$ if we can find a submodule $C\subseteq B$ such that $A\oplus C=B$. 
    We refer to $C$ as a complement of $A$ in $B$. 
\end{definition}

\begin{lemma}\label{lemma:finite_imker_family}
	If $f$ is pointwise free and finitely-generated, then for a fixed $a\in I$, the family $\{\Ker[a,y]\}_{y\in I}$ of kernels out of $f_a$ is finite. 
	Similarly, if the cokernels of structure maps are free, then the family $\{\im[x,a]\}_{x\in I}$ of images into $f_a$ is finite.  
\end{lemma}

We will use the following fact: if $M$ is an $R$-module, and $A\subseteq B\subseteq M$ where $A$ and $B$ are free, then $B/A$ is free if and only if $A$ has a complement in $B$.

\begin{proof}[Proof of Lemma \ref{lemma:finite_imker_family}]
	We first show that the family $\{\Ker[a,y]\}_{y\in I}$ of kernels out of $f_a$ is finite. 
	% Recall that for $y\leq a$, we define $\Ker[a,y]=0$, while for $y>a$ we define $\Ker[a,y]=\ker(f(a\leq y))$.
	Let $y_1,y_2\in I$ such that $a< y_1\leq y_2$. 
    Then we have a short exact sequence 
    \[
    0 \to \Ker[a,y_1] \hookrightarrow  \Ker[a,y_2] \xrightarrow{f(a \le y_1)} Z \to 0
    \]
    for some $Z \subseteq f_{y_1}$. Because $f_{y_1}$ is free, so too is $Z$. 
    Therefore the sequence splits; in particular $\Ker[a,y_1]$ has a complement in $\Ker[a,y_2]$. 
    If there exists an infinite sequence $\Ker[a,y_1] \subsetneq \Ker[a, y_2] \subsetneq \cdots $, then $f_a$ would therefore be infinite-rank --- a contradiction. 
    Thus, the family $\{\Ker[a,y]\}_{y\in I}$ of kernels out of $f_a$ is finite.

	% \textcolor{purple}{[[I found the following hard to understand so proposed an alternative (above)]] We will show that $\Ker[a,y_1]$ is a direct summand of $\Ker[a,y_2]$.
	% By using the first isomorphism theorem on the restricted map $f(a\leq y_1)|_{\ker(f(a\leq y_2))}$, we have  \greg{In particular I don't quite understand where we get the first of the following equations}
	% \begin{align*}
	% 	\ker(f(a\leq y_2)) &= \ker(f(a\leq y_1)|_{\ker(f(a\leq y_2))}) \oplus A \\
	% 	&= \ker(f(a\leq y_1)\cap \ker(f(a\leq y_2)) \oplus A\\
	% 	&= \ker(f(a\leq y_1) \oplus A\,,
	% \end{align*}
	% where $A\cong \im(f(a\leq y_1)|_{\ker(f(a\leq y_2))})$.
	% It therefore follows that for $y_1\leq y_2$, we have that $\Ker[a,y_1]$ is a direct summand of $\Ker[a,y_2]$.
	% Because $\Ker[a,\infty]=f_a$ is finitely generated, it follows that the family $\{\Ker[a,y]\}_{y\in I}$ of kernels out of $f_a$ is finite.}
	
	We now show that the family $\{\im[x,a]\}_{x\in I}$ of images into $f_a$ is finite. 
	% Recall that for $x\geq a$, we define $\im[x,a]= f_a$, while for $x<a$ we define $\im[x,a]=\im(f(x\leq a))$.
	Let $x_1,x_2\in I$ be such that $x_1\leq x_2<a$.
    Then we have a short exact sequence \[
    0 \to \im[x_1, a] \hookrightarrow \im[x_2, a] \xrightarrow{} Z \to 0
    \]
    where $Z = \im[x_2,a] / \im[x_1,a] \subseteq f_a / \im[x_1,a]$. 
    The quotient $f_a / \im[x_1,a]$ is the cokernel of $f(x_1 \le a)$, so it is free by hypothesis. 
    This implies that $Z$ is also free, so the sequence splits, which yields $\im[x_2,a] = \im[x_1,a] + A$ for some complement $A$. Therefore, if there exists infinitely many distinct submodules of form $\im[x,a]$,  the module $f_a$ would have infinite rank --- a contradiction. 
    Thus, the family $\{\im[x,a]\}_{x\in I}$ of images into $f_a$ is finite.
%
 %    \textcolor{purple}{[[similar to above, had trouble understanding so re-wrote]]
	% We will show that $\im[x_1,a]$ is a direct sum of $\im[x_2,a]$. 
	% Because the cokernel of every structure map is free, we can write $f_a = \im(f(x_2\leq a))\oplus A$.
	% By noting that $A\cap \im(f(x_1\leq a))=0$, we have
	% \begin{align*}
	% 	\frac{f_a}{\im(f(x_1\leq a))} &= \frac{\im(f(x_2\leq a))\oplus A}{\im(f(x_1\leq a))}\\ 
	% 	&\cong \frac{\im(f(x_2\leq a))}{\im(f(x_1\leq a))} \oplus A\,.
	% \end{align*}
	% Because the cokernel $f_a/\im(f(x_1\leq a))$ is free, $\frac{\im(f(x_2\leq a))}{\im(f(x_q\leq a))}=\frac{\im[x_2,a]}{\im[x_1,a]}$ must also be free, which implies that $\im[x_1,a]$ is a direct summand of $\im[x_2,a]$.
	% Because $f_a$ is finitely generated, it follows that the family $\{\im[x,a]\}_{x\in I}$ of images into $f_a$ is finite.}
\end{proof}

\subsection{A brief discussion on finitely-indexed persistence modules}\label{sec:finitely-indexed_review}

We briefly review relevant results on interval decompositions of finitely-indexed persistence modules with PID coefficients. 
See \cite{Luo2023} for a more thorough discussion.

Consider the finite totally-ordered poset category $\mathcal{I}=\{0,\ldots,n\}$ with the usual ordering, and let $f:\mathcal{I}\to R\mathrm{-Mod}$ be a persistence module that is pointwise free and finitely-generated, where $R$ is a PID.

If the structure maps of $f$ all have free cokernels, 
then \cite[Theorem 17]{Luo2023} states that, for each pair $(p,q)$ of integers such that $1 \le p, q \le n$, there exits a (non-unique) submodule $A^{pq}_a \subseteq f_a$ such that
\begin{align*}
% \im[{p},a]\cap \Ker[a,q]
\IK_a[p,q]
= &
A^{pq}_a
\oplus
\Bigl(
% (\im[{p-1},a]\cap \Ker[a,q]) + (\im[p,a]\cap \Ker[a,{q-1}])
\IK_a[p-1,q]+\IK_a[p,q-1]
\Bigr)
\end{align*}
Theorem \cite[Theorem 18]{Luo2023} states that, independent of one's choice of $A^{pq}_a$, one has
% \begin{align}
% 	\bigoplus_{p,q} A_{a}^{pq}
%     & = f_a
%     \label{eq:decomp_a_finite}
%     \\
%     \bigoplus_{\substack{q \leq y}} A_a^{pq} 
%     & = \Ker[a, y] \,,    
%     \tag{$p$ runs over $1, \ldots, n$}
%     \\
%     \bigoplus_{\substack{p \leq x}} A_a^{pq} 
%     & = \mathrm{Im}[x, a] \,.    
%     \tag{$q$ runs over $1, \ldots, n$}    
% \end{align}

\begin{align}
\bigoplus_{p,q} A_{a}^{pq} &= f_a\,, \quad
\bigoplus_{q \leq y} A_a^{pq} = \Ker[a, y]\,, \quad
\bigoplus_{p \leq x} A_a^{pq} = \mathrm{Im}[x, a]\,, 
\label{eq:decomp_a_finite}
\end{align}
where $p,q\in\{1,\ldots,n\}$.
% \[
% \bigoplus_{\substack{q \leq y}} A_a^{pq} = \Ker[a, y] \,,
% \]
% \[
% \bigoplus_{\substack{p \leq x}} A_a^{pq} = \mathrm{Im}[x, a] \,.
% \]

% \greg{assumes PID coefficients} If the structure maps of $f$ all have free cokernels, 
% then for each $a\in \mathcal{I}$, we can decompose $f_a$ by 
% \begin{align}
% 	f_{a} &= \bigoplus_{1\leq p,q\leq n} A_{a}^{pq}\,,\label{eq:decomp_a_finite}
% \end{align}
% where $A^{pq}_a$ is any complement of $(\im[{p-1},a]\cap \Ker[a,q]) + (\im[p,a]\cap \Ker[a,{q-1}])$ in $\im[{p},a]\cap \Ker[a,q]$. (Such a complement exists by \cite[Theorem 17]{Luo2023}.)
% One can check that 
% \[
% \bigoplus_{\substack{q \leq y}} A_a^{pq} = \Ker[a, y] \,,
% \]
% \[
% \bigoplus_{\substack{p \leq x}} A_a^{pq} = \mathrm{Im}[x, a] \,.
% \]

In the context of persistent homology, a submodule $A_a^{pq}\subseteq f_a$ in this decomposition can be thought of as cycles that are born at $p$ and die at $q$. 
This is because (1) whenever $a\notin[p,q)$, we must have $A_a^{pq}=0$ (see \cite[Lemma 15]{Luo2023}), and (2) one can choose submodules $A^{pq}_a$ such that $f(a\leq b)(A_a^{pq})=A_b^{pq}$. 
That is, for any choice of $A_a^{pq}$, the image $f(a\leq b)(A_a^{pq})$ is a complement of 
$\IK_b[p-1,q]+\IK_b[p,q-1]$ in $\IK_b[p,q]$
% $(\im[{p-1},b]\cap \Ker[b,q]) + (\im[p,b]\cap \Ker[b,{q-1}])$ in $\im[{p},b]\cap \Ker[b,q]$) 
% $(\im[{p-1},b]\cap \Ker[b,q]) + (\im[p,b]\cap \Ker[b,{q-1}])$ in $\im[{p},b]\cap \Ker[b,q]$) 
(see \cite[Lemma 16]{Luo2023}).

\section{\textbf{Necessity}}
\label{sec:necesity}

Here we prove the necessity direction of Theorem \ref{thm:problem_statement}. 

\begin{theorem}\label{thm:necessity}
    Let $f:I\to R\mathrm{-Mod}$ be a persistence module that is pointwise free and finitely-generated, where $R$ is a PID. 
    If $f$ admits an interval decomposition, then every structure map $f(a\leq b)$ has free cokernel.   
\end{theorem}

\begin{proof}
    It is readily checked that for any persistence modules $h_1,h_2:I\to R\mathrm{-Mod}$, we have $\text{coker}\Bigl( (h_1 \oplus h_2)  (a\le b) \Bigr ) \cong \mathrm{coker} \Bigl( h_1(a\le b) \Bigr) \oplus \mathrm{coker} \Bigl( h_2(a\le b)  \Bigr)$.
    The cokernel of every structure map of an interval module is free, as the image is either $R$ or $0$.
    Thus, the cokernel of every direct sum of interval modules is also free.
\end{proof}

\section{\textbf{Sufficiency}}
\label{sec:sufficiency}

% \greg{This section is devoted to the proof of sufficiency in Theorem \ref{thm:problem_statement}. Formally, we wish to prove the following.}
{We now prove sufficiency of Theorem \ref{thm:problem_statement}, which is the following.}

\begin{theorem}\label{thm:sufficiency}
    Let $R$ be a PID and $f:I\to R\mathrm{-Mod}$ be a persistence module that is pointwise free and finitely-generated. 
    Suppose every structure map $f(a\leq b)$ has free cokernel.   
    Then $f$ admits an interval decomposition.
\end{theorem}

\begin{proof}
    This result depends on observations developed below, synthesized here into a complete proof.
    Suppose that the persistence module $f$ has free cokernels. 
    Lemma \ref{lemma:zorn} implies that $f$ admits a consistent basis (see Definition \ref{def:bases}) if, for every  proper subset $J \subsetneq I$ such that $\{-\infty, \infty\} \subseteq J$ and every consistent basis $(\beta_j)_{j \in J}$ with respect to $f$, the basis can be extended by one element $\beta_a$, for $a \notin J$. 
    This condition holds by Theorem \ref{thm:extension}. Therefore a consistent basis exists.
\end{proof}

\fbox{
\parbox{0.9\textwidth}{

Throughout this section, we 
\begin{itemize}
    \item fix a PID $R$, our coefficent ring;
    % \item fix a PID $R$, which will be the coefficient ring for our persistence modules;
    \item fix a persistence module $f:I\to R\mathrm{-Mod}$ that is \jerry{nonzero (i.e., $f_i\neq0$ for some $i$) and} pointwise free and finitely-generated; 
    \item assume that every structure map 
    % $f(a\leq b):f_a\to f_b$
    has free cokernel.
\end{itemize}

% Suppose that $f: I \to A$ is a persistence module.  
}
}
\\

% ================================================
\subsection{Consistent bases and Zorn's condition} We begin by establishing some common language to talk about bases for the modules of interest.
% ================================================

\begin{definition}
\label{def:bases}
\;
\begin{itemize}
    \item The \emph{restriction} of $f$ to $J$ is the restriction of the functor $f$ to the full subcategory $J \subseteq I$.  
    We denote this object by $f|_J$.  
    \item A \emph{consistent basis} for $f$ is an indexed family $(\basis_i)_{i \in I}$ of bases such that the matrix representation of every structure map with respect to these bases is a matching matrix. (Note that $f$ has a consistent basis if and only if $f$ admits an interval decomposition; we use these terminologies interchangeably.)
    \item A \emph{consistent basis of $f|_J$ \textbf{with respect to $f$}} is a consistent basis $(\basis_j)_{j \in J}$ for $f|_J$, with the additional condition that, for each $j \in J$, the basis $\basis_j$ contains a basis for the submodules $\im[i,j]$ and $\Ker[j,k]$, for all $i,k\in \I$. 
\end{itemize}
\end{definition}

We address Theorem \ref{thm:sufficiency} by examining restrictions of $f$ and their consistent bases with respect to $f$.

\begin{lemma}[Zorn's condition]
	\label{lemma:zorn}
	% The persistence module $f$ admits a consistent basis if the following condition holds: for any set $J$ such that (i) $\{-\infty, \infty\} \subseteq J \subsetneq I$ and (ii) $f|_\J$ admits a consistent basis $(\basis_j)_{j \in \J}$ with respect to $f$, there exists an element $a \in \I \setminus \J$  and a basis $\basis_{a}$ for $f_{a}$ such that ``appending'' $\basis_{a}$ to $(\basis_i)_{i \in J}$ produces a consistent basis $(\basis_i)_{i \in \J\cup\{a\}}$ of $f|_{\J \cup \{a\}}$ with respect to $f$.
    The persistence module $f$ has a consistent basis if the following condition holds: 
    For any set $J$ such that (i) $\{-\infty, \infty\} \subseteq J \subsetneq I$ and (ii) $f|_\J$ admits a consistent basis $(\basis_j)_{j \in \J}$ with respect to $f$, there exists an element $a \in \I \setminus \J$  and a basis $\basis_{a}$ for $f_{a}$ such that ``appending'' $\basis_{a}$ to $(\basis_i)_{i \in J}$ produces a consistent basis $(\basis_i)_{i \in \J\cup\{a\}}$ of $f|_{\J \cup \{a\}}$ with respect to $f$.
\end{lemma}
\begin{proof}
	We apply Zorn's lemma. 
	Let $P$ be the set whose elements are pairs $(\K, (\beta_k)_{k \in \K})$ such that $\{-\infty, \infty\} \subseteq \K \subseteq \I$ and $(\beta_k)_{k \in \K}$ is a consistent basis of $f|_\K$ with respect to $f$. Impose a partial order on $P$ such that $(\K, (\beta_k)_{k \in \K}) \le (\K', (\beta'_k)_{k \in \K'})$ if and only if $\K \subseteq \K'$ and $\beta_k = \beta'_k$ for all $k \in \K$. 
    Note that $P$ is nonempty, as we can take $K=\{-\infty,\infty\}$ with $\beta_{-\infty} = \beta_\infty = \emptyset$ (recall that $f_{-\infty}=f_\infty = 0$ by convention).
	Every chain in this poset clearly has an upper bound in $P$, as we can take the union of indexing sets. 
	Therefore, by Zorn's lemma, $P$ contains a maximal element. 
	
	% If the condition stated in the lemma statement holds true, then $(\K, (\beta_k)_{k \in \K})$ can only be maximal in $P$ if $\K = \I$, because if $\K\subsetneq \I$, we can find $a\in\I\setminus\K$ and a basis $\beta_{a}\subseteq f_{a}$ such that $(\K\cup\{a\},(\beta_k)_{k\in \K\cup\{a\}})\in P$. 
	% Therefore, if the condition in the lemma statement is true then there must exist a consistent basis for $f$.
    Suppose, now, that the condition holds. 
    Then $(\K, (\beta_k)_{k \in \K})$ can only be maximal in $P$ if $\K = \I$, because if $\K\subsetneq \I$, we can find $a\in\I\setminus\K$ and a basis $\beta_{a}\subseteq f_{a}$ such that $(\K\cup\{a\},(\beta_k)_{k\in \K\cup\{a\}})\in P$. 
	Therefore $f$ has a consistent basis, as desired.
\end{proof}

\subsection{Partitioning consistent bases with images and kernels}
\label{sec:basis_partition}
% ================================================

We now partition a consistent basis $\beta_j$ into disjoint subsets $\beta_j^{pq}$, where each subset contains basis elements with the same ``birth values'' and ``death values''.
% In this subsection, we provide some general notation for partitioning the elements of a consistent basis $\beta_j$ as a disjoint union of subsets $\beta_j^{pq}$, where all the elements of $\beta_j^{pq}$ have similar ``birth'' and ``death'' values.

\begin{lemma}\label{lemma:basis_subset_span}
    Let $M$ be a free and finitely-generated $R$-module, and let $\beta\subseteq M$ be a basis. 
    Let $\mathcal{M}_\beta$ be the family of submodules of $M$ that can be written as the span of a subset of $\beta$. 
    That is 
    $\mathcal{M}_\beta = \{ \mathrm{span}(\gamma) : \gamma\subseteq \beta\}$.
    The family $\mathcal{M}_\beta$ is closed under sum and intersection. 
\end{lemma}
\begin{proof}
    Without loss of generality, we can assume $M=R^n$, and $\beta$ consists of standard unit vectors. 
    Let $N_1,N_2\in\mathcal{M}_\beta$, in which $N_1=\mathrm{span}(\gamma_1)$ and $N_2=\mathrm{span}(\gamma_2)$. 
    It is readily checked that $N_1 + N_2 = \mathrm{span}(\gamma_1 + \gamma_2)$ and $N_1\cap N_2 = \mathrm{span}(\gamma_1\cap\gamma_2)$. 
\end{proof}

Lemma \ref{lemma:basis_subset_span} implies the following property for consistent bases.

\begin{cor}\label{cor:basis_subset_basis}
    If a basis $\beta_j$ for $f_j$ contains a basis for every image into and kernel out of $f_j$, then $\beta_j$ also contains a basis for any submodule of $f_j$ that can be obtained by taking sums and intersections of submodules of form $\im[i,j]$ and $\Ker[j,k]$. 
\end{cor}

Fix $J \subsetneq I$ and $a \in I \setminus J$.
% Choose a finite index set $S=\{\idx_1,\ldots,\idx_r\}\subseteq I$ such that, for every $c\in I$, we have $\Ker[a,c]=\Ker[a,\idx_t]$ for some $t$ and $\im[c,a]=\im[\idx_s,a]$ for some $s$. 
% Such a finite set exists by Lemma \ref{lemma:finite_imker_family}.
% % A finite set of this form exists because there exist only finitely many distinct submodules of form $\im[p,j]$ and $\Ker[j,q]$, by Lemma \ref{lemma:finite_imker_family}.
% We further require $\{\idx_1,\ldots,\idx_r\}$ to (1) contain $a$, (2) be of minimal size, and (3) contain as many elements of $J$ as possible. 
% (That is, for any $j\in J$ such that $j\geq a$, there must be some $\idx_q\in J$ such that $\Ker[a,j]=\Ker[a,\idx_q]$. Similarly, for any $j\in J$ such that $j\leq a$, there must be some $\idx_p\in J$ such that $\im[j,a]=\im[\idx_p,a]$.)
% % (That is, for any $j\in J$ such that $j\geq a$, there must be some $p$ such that $\idx_p\in J$ and $\Ker[a,j]=\Ker[a,\idx_p]$. Similarly, for any $j\in J$ such that $j\leq a$, there must be some $q$ such that $\idx_q\in J$ and $\im[j,a]=\im[\idx_q,a]$.) 
\jerry{Fix a finite set $S_J$ such that $\{a \} \subseteq S_J \subseteq J \cup \{a\}$, which ``covers'' all images and kernels of $J$ in $f_a$, in the sense that $\{\im[j,a]: j \in J\cup\{a\}\} = \{\im[s,a]: s \in S_J\}$ and $\{\Ker[a,j]: j \in J\cup\{a\}\} = \{\Ker[a,s]: s \in S_J\}$. Such a finite set exists, by Lemma \ref{lemma:finite_imker_family}.  Next, fix a finite set $S$  such that $S_J \subseteq S \subseteq I$, which covers all images and kernels of $I$ in $f_a$, in the sense that $\{\im[i,a]: i \in I\} = \{\im[s,a]: s \in S\}$ and $\{\Ker[a,i]: i \in I\} = \{\Ker[a,s]: s \in S\}$. Place the elements of $S$ into an ordered sequence $i_1 < \cdots < i_r$.}

% \greg{Suggest reword: instead of "with resepct to f" say something like "with respect to I", or "with respect to $f|_I$", or "compatible with $f$ on $I$" to emphasize the connection with the ambient poset.}
% \jerry{"compatible with $f$ on $I$" sounds good! consider a triplet $(f,J,I)$??? }  \textcolor{purple}{Suggest calling this a "strongly consistent basis for $f$ on $J$"}

Supposing that one exists, fix a consistent basis $(\beta_j)_{j \in J}$ for $f|_J$ with respect to $f$. For each $j \in J$, define functions $\mathcal{K}: f_j \to \Z$ by $ \mathcal{K}(v) = \min \{ q : v \in \Ker[j, \idx_q]\}$ and $\mathcal{I}: f_j \to \Z$ by $\mathcal{I}(v) = \min \{ p : v \in  \im[\idx_p, j]\}$. 
Then, for all $j \in J$ and all $(p,q) \in \{1, \ldots, r\} \times \{1, \ldots, r\}$, define
% \begin{align*}
%     \beta^{pq}_{j} &: = \{ v \in \beta_{j} : (\mathcal{I}(v),\mathcal{K}(v)) = (p,q) \}\,, 
%     \\
%     A^{pq}_j & : = \textrm{span}( \beta^{pq}_j)\,.
% \end{align*}
\begin{align*}
    \beta^{pq}_{j} : = \{ v \in \beta_{j} : (\mathcal{I}(v),\mathcal{K}(v)) = (p,q) \}\,, 
    &&
    A^{pq}_j  : = \textrm{span}( \beta^{pq}_j)\,.
\end{align*}

Several observations are immediate, for any $j \in J$.
\begin{itemize}
    \item Basis $\beta_j$ partitions as a disjoint union $\beta_j = \bigsqcup_{p,q} \beta_j^{pq}$.
    \item Consequently, $f_j = \bigoplus_{p,q}A^{pq}_j$.
    \item For each $\idx_n$ the set $\bigcup_{q \le n}\beta^{pq}_j = \{ v \in \beta_j : \mathcal{K}({v}) \le n\} = \beta_j \cap \Ker[j,\idx_n]$ forms a basis for $\Ker[j,\idx_n]$. Recall that $\beta_j$ contains a basis for $\Ker[j,\idx_n]$ by hypothesis (see Definition \ref{def:bases}).
    \item Similarly, for each $\idx_n$, the set $\bigcup_{p \le n}\beta^{pq}_j = \{ v \in \beta_j : \mathcal{I}({v}) \le n\} = \beta_j \cap \im[\idx_n, j]$ forms a basis for $\im[\idx_n,j]$. Recall that $\beta_j$ contains a basis for $\im[\idx_n, j]$ by hypothesis (see Definition \ref{def:bases}). 
    \item If $V \subseteq f_j$ is any submodule which can be obtained via a finite sequence of sums and intersections of submodules of the form $\im[\idx_m,j]$ and $\Ker[j, \idx_n]$, then $V \cap \beta_j$ is a disjoint union of subsets of form ${\beta}^{pq}_j$. 
    Moreover, $V \cap \beta_j$ is a basis for $V$. This observation follows from Lemma \ref{lemma:basis_subset_span} and Corollary \ref{cor:basis_subset_basis}.
    
    \item If $X$ and $Y$ are disjoint subsets of $\{1, \ldots, r\} \times \{1, \ldots, r\}$, then $\textrm{span}(\bigsqcup_{(p,q) \in X \cup Y} \beta^{pq}_j) = \textrm{span}(\bigsqcup_{(p,q) \in X} \beta^{pq}_j) \oplus \textrm{span}(\bigsqcup_{(p,q) \in Y} \beta^{pq}_j)$
    \item As a special case, we have 
    \begin{align}
        % \im[\idx_{p},a]\cap \Ker[a,\idx_q]
        \IK_j[\idx_p,\idx_q]
         & =   
         \textrm{span} \left ( \bigcup_{m \le p, n \le q} \beta_j^{pq} \right ) \nonumber
         \\
         % & = 
         % \textrm{span}(\beta^{pq}_j)
         % \oplus 
         % \textrm{span}\Biggl( \bigcup_{m \le p, n \le q, (m,n) \neq (p,q)} \beta_j^{pq} \Biggr )\nonumber         
         % \\
         & = 
         \textrm{span}(\beta^{pq}_j)
         \oplus 
         \textrm{span}\Biggl( 
            \left(  \cup_{m < p, n \le q} \beta_j^{pq} \right )  
            \bigcup 
            \left(  \cup_{m \le p, n < q} \beta_j^{pq} \right )            
        \Biggr )\nonumber
         \\
         & = 
         \textrm{span}(\beta^{pq}_j)
         \oplus 
         \Biggl( 
            \textrm{span}\left(  \cup_{m < p, n \le q} \beta_j^{pq} \right )  
            +
            \textrm{span}\left(  \cup_{m \le p, n < q} \beta_j^{pq} \right )            
        \Biggr )\nonumber
        \\
        & = 
        A^{pq}_j
        \oplus 
        \Bigl( 
            \IK_j[\idx_{p-1},\idx_q] + \IK_j[\idx_p,\idx_{q-1}]
        \Bigr) \label{eq:A_j^pq_complement}
    \end{align}
\end{itemize}

\subsection{Incremental Extension}

Lemma \ref{lemma:zorn} implies that, to show Theorem \ref{thm:sufficiency}, it is enough to prove the following.

\begin{theorem}[Incremental Extension]
\label{thm:extension}
    Let $J\subsetneq I$ be a proper subset containing $-\infty$ and $\infty$, and suppose $f|_J$ admits a consistent basis $(\beta_j)_{j\in J}$ with respect to $f$. 
    Then, for each $a\in I\setminus J$, there exists a basis $\beta_{a}\subseteq f_a$ such that $(\beta_j)_{j\in J\cup\{a\}}$ forms a consistent basis for $f|_{J\cup\{a\}}$ with respect to $f$. 
\end{theorem}
\begin{proof}
    In the discussion below we construct a basis $\beta_a$, and show that $(\beta_j)_{j\in J\cup\{a\}}$ forms a consistent basis for $f|_{J\cup\{a\}}$ with respect to $f$ (Theorem \ref{thm:consistent_basis}).
\end{proof}

% \subsection{Proving Theorem \ref{thm:extension}}

% The remainder of this section is devoted to proving Theorem \ref{thm:extension}, which is accomplished by Theorem \ref{thm:consistent_basis}. \greg{An overall summary of the proof of Theorem \ref{thm:extension} appears at the end of this section.}

The remainder of this section is devoted to the details of the proof of Theorem \ref{thm:extension}. 
We will first construct the desired basis $\beta_a \subseteq f_a$ by patching together images and inverse images of subsets of the bases $\{\beta_j\}_{j\in J}$, which already exist over $J$. 
We will then show that $\beta_a$ has the required properties.

Fix a subset $J$ such that $\{-\infty, \infty\} \subseteq J \subsetneq I$, and choose some $a\notin J$.
Choose a (finite) index set $S = \{\idx_1, \ldots, \idx_r\}$ as directed in Section \ref{sec:basis_partition}. 
Additionally, choose indices $s < k< t$ such that 
(i) $\idx_t\in J\cap S$ is the smallest index such that $\idx_t>a$,
(ii) $\idx_s\in J\cap S$ is the largest index such that $\idx_s<a$, and
(iii) $\idx_k = a$.
% \begin{itemize}
%     \item $\idx_t\in J\cap S$ is the smallest index such that $\idx_t>a$;
%     \item $\idx_s\in J\cap S$ is the largest index such that $\idx_s<a$;
%     \item $\idx_k = a$.
% \end{itemize}
Note that $\idx_s$ and $\idx_t$ must exist because $J$ contains $-\infty$ and $\infty$.

\begin{remark}
    Because $S$ is finite, the results from \cite{Luo2023} apply to our analysis of $f|_S$.
\end{remark}

% Next, we will define a basis $\beta_a\subseteq f_a$ by defining $\beta^{pq}_a$ and $A^{pq}_a$ in a consistent manner with the bases $\{\beta_j\}_{j\in J}$; we will then take the union to define $\beta_a=\bigcup_{p,q}\beta^{pq}_a$.

% Because $a \notin J$, the sets $\beta^{pq}_a$ cannot be defined in the same way as $j\in J$.
% We will define $\beta^{pq}_a$ as a basis for $A^{pq}_a$ which, in turn, is a complement of $\IK_a[i_{p-1},i_{q}]+\IK_a[i_{p},i_{q-1}]$ in $\IK_a[i_p,i_q]$.
% By considering the following four cases, we will define our $\beta^{pq}_a$ as follows.

% We now check that our choices of $\beta^{pq}_a$ are valid. 
% That is, we check that $A^{pq}_a = \mathrm{span}(\beta^{pq}_a)$ is a complement of $\IK_a[i_{p-1},i_{q}]+\IK_a[i_{p},i_{q-1}]$ in $\IK_a[i_p,i_q]$.
% Note that we only need to check this for cases (1)---(3), as case (4) follows by definition.

Partition $\beta_j$ into disjoint subsets $\beta_j^{pq} \subseteq \beta_j$, and define $A_j^{pq} = \textrm{span}(\beta^{pq}_j) \subseteq f_j$ as directed in Section \ref{sec:basis_partition}, for each $j \in J$. 
  Intuitively, $\beta_j^{pq}$ represents a subset of $\beta_j$ consisting of basis vectors with the same ``birth value'' (represented by the integer $p$) and ``death value'' (represented by the integer $q$).

Now define corresponding subsets $\beta^{pq}_a \subseteq f_a$ as follows.  
See Figure \ref{fig:cases_timeline_picture} for reference.
\begin{enumerate}
    \item If $i_p>a$ or $i_q\leq a$, we define $\beta^{pq}$ to be empty. This makes intuitive sense, because $p$ and $q$ represent ``birth values'' and ``death values'', respectively; no vectors in $f_a$ are born later than $a$ or die before $a$.
    \item For $\idx_p \le \idx_s$ we will define $\beta^{pq}_a = f(\idx_s \leq a)(\beta^{pq}_{\idx_s})$. 
    That is, if the ``birth value''  $i_p$ occurs before $i_s$, then we obtain a set of basis vectors $\beta^{pq}_a$ by pushing forward some basis vectors in $\beta_{i_s}$. 
\end{enumerate}
The remaining case is $i_s < i_p \le a$. We divide this into the two subcases where $i_t < i_q$ and $i_t \ge i_q$, as follows:
\begin{enumerate}
    \setcounter{enumi}{2}
    \item For $i_s< i_p \leq a < i_t < i_q$, we define $\beta^{pq}_a \subseteq f(a \leq \idx_t)^{-1}(\beta^{pq}_{\idx_{t}})$ by choosing one representative from the preimage of each basis element $v \in \beta^{pq}_{\idx_s}$. 
    That is, if the ``death value'' $\idx_q$ occurs {strictly} after $\idx_t$, then we obtain a set of basis vectors $\beta^{pq}_a$ by pulling back some basis vectors in $\beta_{\idx_t}$. 
    We require each preimage vector to lie in the submodule $\IK_a[\idx_p, \idx_q]$.
    Note that this requirement can always be satisfied because each $v \in \beta^{pq}_{\idx_t}$ lies in the image of $f(\idx_p \le \idx_t)$, hence also in the image of $f(a \leq \idx_t)$.  
    \item If none of the preceding rules applies (i.e., $\idx_s < \idx_p \le a <  \idx_q \leq \idx_t$), we choose any complement $A_a^{pq}$ of $\IK_a[i_{p-1},i_{q}]+\IK[i_{p},i_{q-1}]$ in $\IK_a[i_p,i_q]$, and define $\beta^{pq}_a$ to be any basis of $A_a^{pq}$.  
    (Such a complement exists by \cite[Theorem 17]{Luo2023}.)
    In this case, there is nothing in $J$ to push forward or pull back to index $a$, so we choose an arbitrary complement.
\end{enumerate}

\begin{lemma}
    The set $\beta^{pq}_a$ is linearly independent, for all $p$ and $q$.
\end{lemma}
\begin{proof}
    In case (1), $\beta^{pq}_a$ is empty, so there is nothing to prove. 
    
    % In case (2), suppose for a contradiction that $\sum_{v \in \beta^{pq}_a} \alpha_v \cdot v = 0$ for some family of coefficients $(\alpha_v)$, at least one of which is nonzero. If we apply the same linear combination to the corresponding elements of $\beta_{\idx_s}^{pq}$, then the resulting linear combination must lie in $\Ker[\idx_s,a]$. However this is a contradiction, since the basis $\beta_{\idx_s}$ contains a basis for $\Ker[\idx_s,a]$ which is disjoint from $\beta_j^{pq}$ -- in particular, it would imply that two different linear combinations of elements in $\beta_j$ equal the same vector. 
    In case (2), we have, by \cite[Lemma 16]{Luo2023}, that $f(i_s\leq a)|_{{A_{\idx_s}^{pq}}}$ restricts to an isomorphism onto its image. 
    Because $\beta^{pq}_{\idx_s}$ is linearly independent (it is a basis of $A^{pq}_{\idx_s}$), its image $f(i_s\leq a)(\beta^{pq}_{\idx_s})$ is as well.
    
    In case (3), the elements of $\beta^{pq}_a$ map bijectively onto a linearly independent set $\beta^{pq}_{\idx_t}$, under a linear map; therefore, $\beta^{pq}_a$ is linearly independent. 
    
    In case (4), the set $\beta^{pq}_a$ is linearly independent by hypothesis. 
    % This completes the proof.
\end{proof}

\begin{lemma}\label{lemma:valid_A^pq}
    Let $A^{pq}_a: = \mathrm{span}(\beta^{pq}_a)$.  Then, for all $1 \le p, q \le n$ we have
    \begin{align*}
        \IK_a[i_p,i_q]
        =
        A^{pq}_a 
        \oplus 
        \Biggl(
            \IK_a[i_{p-1},i_{q}]+\IK_a[i_{p},i_{q-1}]
        \Biggr )\,.
    \end{align*}
    % In particular, $A^{pq}_a$ complements $\IK_a[i_{p-1},i_{q}]+\IK_a[i_{p},i_{q-1}]$ in the submodule $\IK_a[i_p,i_q]$.
\end{lemma}
\begin{proof}
    We check the four cases defined in the construction of $\beta^{pq}_a$.
    % We check this claim for cases (1)---(3) of the four cases considered in the construction of $\beta^{pq}_a$; the desired conclusion holds  by definition in case (4).
    % \\

\noindent\textit{Case (1)} Suppose that $i_p>a$ or $i_q\leq a$. 
By the proof of \cite[Lemma 15]{Luo2023}, we have that $\IK_a[i_{p-1},i_{q}]+\IK_a[i_{p},i_{q-1}]=\IK_a[i_p,i_q]$. 
This implies that the complement is $0$, with basis $\beta^{pq}_a := \emptyset$, as desired.
% \greg{This implies that the only complement is $0$, which equals the span of $\beta^{pq}_a = \emptyset$, as desired.}
% \jerry{
% This implies that any complement is trivial and has empty basis. 
% Therefore, $A_a^{pq} = \mathrm{span}(\beta^{pq}_a) = 0$ is a complement. 
% }
% This implies that any complement $A^{pq}_a$ is trivial. 
% Moreover, the basis $\beta^{pq}_a$ is empty, by construction. 
% Therefore $A_a^{pq} = 0 = \mathrm{span}(\beta^{pq}_a)$, as desired.
% \\

\noindent\textit{Case (2)} Suppose that $p\leq s$. 
By \cite[Lemma 16]{Luo2023}, $f(i_s\leq a)(A_{i_s}^{pq})$ is a complement of $\IK_a[i_p,i_{q-1}]+\IK_a[i_{p-1},i_q]$ in $\IK_a[i_p,i_q]$.
% $\greg{[[COMMENT: the following statement isn't needed for the proof and can be removed]]}, and moreover, the morphism $f(i_s\leq a)|_{A^{pq}_{i_s}}$ restricts to an isomorphism onto its image.
% Therefore, it follows that $\beta_a^{pq}=f(i_s\leq a)(\beta_{i_s}^{pq})$ is a basis of $A^{pq}_a=f(i_s\leq a)(A_{i_s}^{pq})$.
% \\

\noindent\textit{Case (3)} Suppose that $i_s< i_p \leq a < i_t < i_q$. 
We first recall from \eqref{eq:A_j^pq_complement} that $\IK_{\idx_t}[\idx_{p},\idx_q] = A^{pq}_{\idx_t} \oplus \Bigl( \IK_{\idx_t}[\idx_{p-1},\idx_q]+\IK_{\idx_t}[\idx_{p},\idx_{q-1}] \Bigr)$. 
Note that $A_a^{pq}$ can be realized as $\ell(A^{pq}_{{i_t}})$, where $\ell$ is some choice of split map in the short exact sequence in Lemma \ref{lemma:short_exact_sequence}. 
Therefore, by Lemma \ref{lemma:short_exact_sequence}, we have 
\begin{align*}
    \IK_a[\idx_p,\idx_q] &= \ell(A_{{i_t}}^{pq}) \oplus \Bigl( \IK_{a}[\idx_{p-1},\idx_q]+\IK_{a}[\idx_{p},\idx_{q-1}] \Bigr)\\
    &= A_a^{pq} \oplus \Bigl( \IK_{a}[\idx_{p-1},\idx_q]+\IK_{a}[\idx_{p},\idx_{q-1}] \Bigr)
\end{align*}

\noindent\textit{Case (4)} This case holds by definition.

As all cases have been addressed, we conclude the proof.
\end{proof}

\begin{figure}[h]
    \centering
    \includegraphics[width=0.9\textwidth]{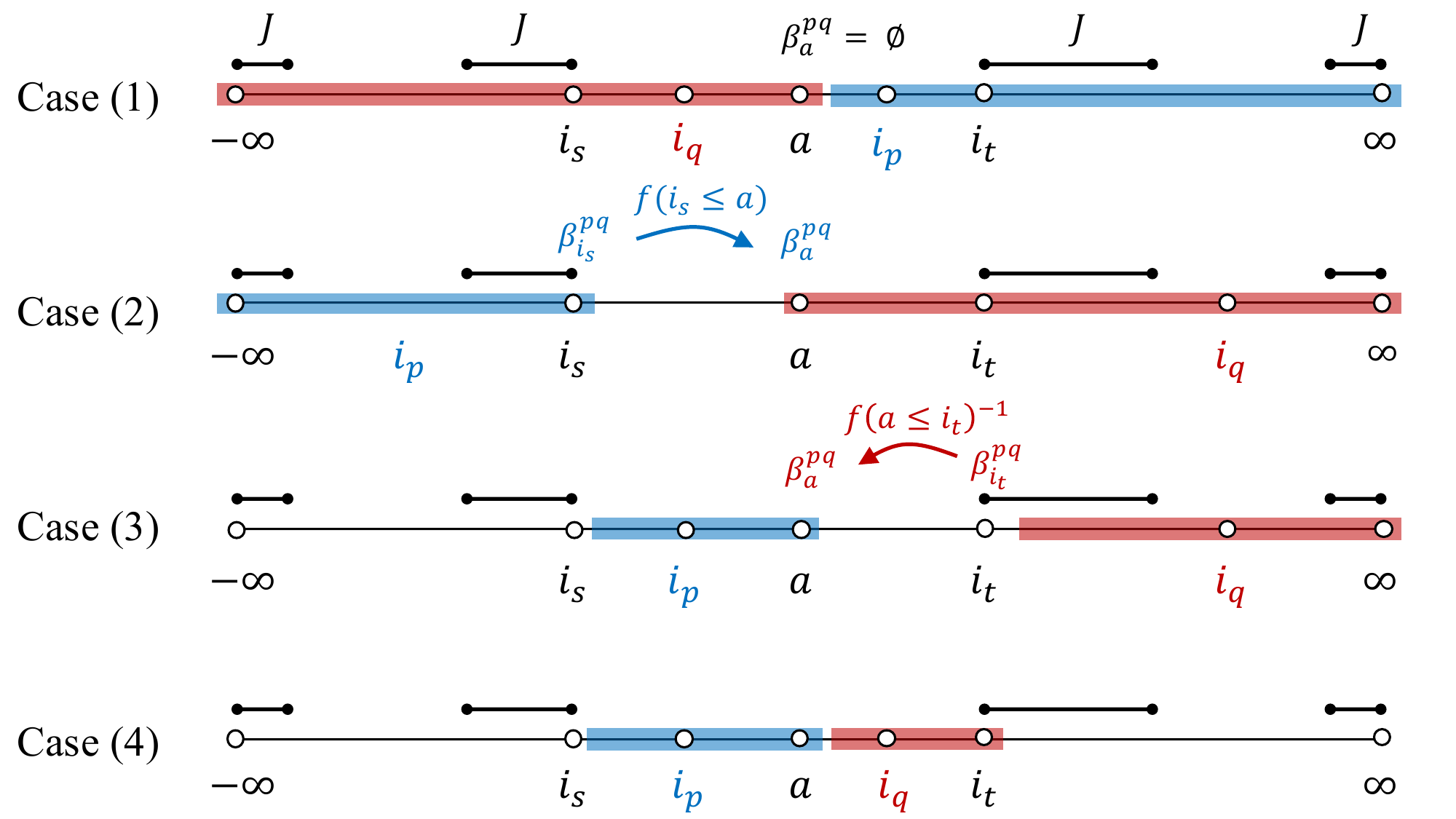}
    \caption{
    % The four cases for defining $\beta_a^{pq}$, for different values of $\idx_p$ and $\idx_q$. 
    % \jerry{In case (1), the $\idx_q<\idx_p$, so the resultant basis is empty. 
    % Cases (2)--(4) concern when $\idx_p\leq \idx_q$, in which we consider the order of $\idx_p$ and $\idx_q$ in relation to elements in $J$.}
    \jerry{The four cases for defining $\beta_a^{pq}$, for different values of $\idx_p$ (birth) and $\idx_q$ (death). Red and blue segments show possible value ranges for $i_q$ and $i_p$, in each case. In Case (1) death occurs before birth, so $\beta_a^{pq}$ must be empty. In Case (2) birth occurs at or before $i_s$, so we may push some vectors from $\beta^{pq}_{i_s}$ forward to form $\beta^{pq}_a$. In Case (3) death occurs strictly after $i_t$, so we may pull some vectors back from $\beta^{pq}_{i_t}$. In Case (4) birth (respectively, death) occurs too late (respectively, too early), so no basis may be obtained from pushing forward or pulling back, and we must construct a new basis entirely.}
    }
    % \greg{Cases defining $\beta_a^{pq}$ for all values of $\idx_p$ and $\idx_q$ (also used in the proof of Lemma \ref{lemma:valid_A^pq}). 
    % Cases (2)---(4) collectively cover all possible values of $\idx_p$ and $\idx_q$ that are excluded by Case (1).
    % 
    \label{fig:cases_timeline_picture}  % Optional: for cross-referencing
\end{figure}

\begin{lemma}\label{lemma:actual_basis}
    The set $\beta_a := \bigsqcup_{1\leq p,q\leq r}\beta^{pq}_a$ forms a basis for $f_a$. 
\end{lemma}

\begin{proof}
    To show that $\beta_a = \bigsqcup_{p,q}\beta^{pq}_a$ is a basis for $f_a$, it is enough to show that $f_a = \bigoplus_{p,q}A^{pq}_a$. 
    This is given by \cite[Theorem 18]{Luo2023}.
\end{proof}

\begin{lemma}\label{lemma:matching_matrix}
    The matrix representation $[f(\idx_s\leq a)]_{\beta_{\idx_s}}^{\beta_a}$ with respect to bases $\beta_{\idx_s}$ and $\beta_a$ is a matching matrix. 
    Similarly, the matrix representation $[f(a\leq \idx_t)]_{\beta_a}^{\beta_{\idx_t}}$ with respect to bases $\beta_a$ and $\beta_{\idx_t}$ is a matching matrix. 
\end{lemma}

\begin{proof}
    We first show that $[f(\idx_s\leq a)]_{\beta_{\idx_s}}^{\beta_a}$ is a matching matrix. 

    For each $p$ and $q$, one of the following mutually exclusive cases must hold: 
    \begin{enumerate}
        \item $\idx_p > \idx_s$ or $\idx_q\leq \idx_s$. In this case, the basis $\beta^{pq}_{\idx_s}$ is empty, by \cite[Lemma 15]{Luo2023}.  
        \item $i_p \le i_s < i_q$. We divide this into two subcases: $a < i_q$ and  $a \ge i_q (> i_s)$
        \begin{enumerate}
            \item  $a < \idx_q$.
            % $\idx_p\leq \idx_s < a < \idx_q$. 
            In this case, the structure map $f(\idx_s\leq a)$ maps $A^{pq}_{\idx_s}$ isomorphically onto $A^{pq}_a$ and, in particular, maps $\beta^{pq}_{\idx_s}$ bijectively onto $\beta_a^{pa}$. 
            \item $\idx_s < \idx_q\leq a$. In this case, $A^{pq}_{\idx_s}\subseteq \Ker[\idx_s,a]$, so the corresponding basis $\beta_{\idx_t}^{pq}$ maps to $0$. 
        \end{enumerate}
    \end{enumerate}
    
    % (A) $\idx_p > \idx_s$ or $\idx_q\leq \idx_s$. In this case the basis $\beta^{pq}_{\idx_s}$ is empty, by \cite[Lemma 15]{Luo2023}. 
    % (B)  $\idx_p\leq \idx_s < a < \idx_q$. In this case the structure map $f(\idx_s\leq a)$ maps $A^{pq}_{\idx_s}$ isomorphically onto $A^{pq}_a$ and, in particular, maps $\beta^{pq}_{\idx_s}$ bijectively onto $\beta_a^{pa}$. 
    % (C) $\idx_s\leq \idx_q\leq a$. In this case $A^{pq}_{\idx_s}\subseteq \Ker[\idx_s,a]$, so the corresponding basis $\beta_{\idx_t}^{pq}$ maps to $0$. 
    Thus, in each case, the basis $\beta^{pq}_{\idx_s}$ is empty, maps to zero, or maps bijectively onto $\beta^{pq}_{a}$. 
    Because $\beta_a = \bigsqcup_{p,q} \beta^{pq}_{a}$ and $\beta_{\idx_s} = \bigsqcup_{p,q}  \beta^{pq}_{\idx_s}$, it follows that the map $[f(\idx_s\leq a)]_{\beta_{\idx_s}}^{\beta_a}$ is a matching matrix.

    % \textcolor{blue}{Recall from \cite[Lemma 15]{Luo2023} that whenever $\idx_p > \idx_s$ or $\idx_q\leq \idx_s$, we have $A^{pq}_{\idx_s}=0$; as such, the basis $\beta^{pq}_{\idx_s}$ would be empty. }
    
    % \textcolor{blue}{For values of $p$ and $q$ such that $\idx_p\leq \idx_s < a < \idx_q$, the structure map $f(\idx_s\leq a)$ maps $A^{pq}_{\idx_s}$ isomorphically onto $A^{pq}_a$ and, in particular, maps $\beta^{pq}_{\idx_s}$ bijectively onto $\beta_a^{pa}$.}
    
    % \textcolor{blue}{For values of $p$ and $q$ such that $\idx_s\leq \idx_q\leq a$, note that $A^{pq}_{\idx_s}\subseteq \Ker[\idx_s,a]$, which implies that the corresponding basis $\beta_{\idx_t}^{pq}$ maps to $0$.}
    
    % \textcolor{blue}{Therefore, we have that $\beta_{\idx_s}\setminus \beta_a^{pq}$ maps injectively into $\beta_a$, which implies the matrix representation $[f(a\leq b)]_{\beta_{\idx_s}}^{\beta_a}$ is a matching matrix. }

    Next, we show that $[f(a\leq \idx_t)]_{\beta_a}^{\beta_{\idx_t}}$ is a matching matrix. For each $p$ and $q$, one of the following mutually exclusive cases must hold.
    \begin{enumerate}
        \item $\idx_p > a$ or $\idx_q\leq a$. In this case, $A^{pq}_a=0$ by definition, with basis $\beta^{pq}_{a}=\emptyset$. 
        % The basis $\beta^{pq}_{a}$ is empty. 
        \item $\idx_q >\idx_t  (> a)$. 
        We divide this case into the subcases where $\idx_s \ge \idx_p$ and where  $\idx_s < \idx_p (\le a)$.
        \begin{enumerate}
            \item $\idx_s\geq \idx_p$. In this case, recall that $A_a^{pq} = f(\idx_s\leq a)(A^{pq}_{\idx_s})$, with basis $\beta^{pq}_a=f(a\leq b)(\beta^{pq}_{\idx_s})$, by construction. 
            Because $(\beta_j)_{j\in J}$ is a consistent basis with respect to $f$, we have $f(\idx_s\leq \idx_t)$ mapping $\beta_{\idx_s}^{pq}$ bijectively onto $\beta_{\idx_t}^{pq}$. 
            Therefore, functoriality of $f$ implies that $f(a\leq \idx_t)$ maps $\beta_{a}^{pq}$ bijectively onto $\beta_{\idx_t}^{pq}$.
            \item $\idx_s<\idx_p(\leq a)$. 
            In this case, recall that we constructed the basis $\beta_a^{pq}\subseteq A_a^{pq}$ by taking the basis $\beta_{\idx_t}^{pq}\subseteq A_{\idx_t}^{pq}$ and choosing a preimage for each element. 
            This implies that $f(a\leq \idx_t)$ maps $\beta_{a}^{pq}$ bijectively into $\beta_{\idx_t}^{pq}$. 
        \end{enumerate}
    \item  $\idx_t\geq \idx_q > a$. 
    Note that $A^{pq}_a\subseteq \Ker[a,\idx_t]$. 
    Therefore, $f(a\leq \idx_t)$ maps the corresponding basis $\beta_{a}^{pq}$ to $0$. 
    \end{enumerate}
    % (A) $\idx_p > a$ or $\idx_q\leq a$. In this case, $A^{pq}_a=0$ by definition; as such, the basis $\beta^{pq}_{a}$ is empty. 
    % (B) $\idx_p\leq \idx_s$ and $\idx_q>\idx_t$. In this case, recall that $A_a^{pq} = f(\idx_s\leq a)(A^{pq}_{\idx_s})$, with basis $\beta^{pq}_a=f(a\leq b)(\beta^{pq}_{\idx_s})$. 
    % Because $(\beta_j)_{j\in J}$ is a consistent basis with respect to $f$, we have $f(\idx_s\leq \idx_t)$ mapping $\beta_{\idx_s}^{pq}$ bijectively into $\beta_{\idx_t}^{pq}$. 
    % Therefore, functoriality of $f$ implies that $f(a\leq \idx_t)$ maps $\beta_{a}^{pq}$ bijectively into $\beta_{\idx_t}^{pq}$.
    % (C) $\idx_s<\idx_p\leq a$ and $\idx_q>\idx_t$. 
    % In this case, recall that we constructed the basis $\beta_a^{pq}$ of $A_a^{pq}$ by taking the basis $\beta_{\idx_t}^{pq}$ of $A_{\idx_t}^{pq}$ and choosing a preimage for each element. 
    % This naturally implies that $f(a\leq \idx_t)$ maps $\beta_{a}^{pq}$ bijectively into $\beta_{\idx_t}^{pq}$. 
    
    Thus, in each case, the basis $\beta^{pq}_{a}$ is empty, maps to zero, or maps bijectively onto $\beta^{pq}_{\idx_t}$. 
    Because $\beta_a = \bigsqcup_{p,q} \beta^{pq}_{a}$ and $\beta_{\idx_t} = \bigsqcup_{p,q}  \beta^{pq}_{\idx_t}$, it follows that the map $[f(\idx_s\leq a)]_{\beta_{\idx_s}}^{\beta_a}$ is a matching matrix.
    %
    % \textcolor{blue}{Now, consider the matrix representation $[f(a\leq \idx_t)]_{\beta_a}^{\beta_{\idx_t}}$.
    % Recall that whenever $\idx_p > a$ or $\idx_q\leq a$, we chose $A^{pq}_a=0$; as such, the basis $\beta^{pq}_{a}$ would be empty. }
    %
    % \textcolor{blue}{For values of $p$ and $q$ such that $\idx_p\leq \idx_s$ and $\idx_q>\idx_t$, recall that $A_a^{pq} = f(\idx_s\leq a)(A^{pq}_{\idx_s})$, with basis $\beta^{pq}_a=f(a\leq b)(\beta^{pq}_{\idx_s})$. 
    % Because $(\beta_j)_{j\in J}$ is a consistent basis with respect to $f$, we have $f(\idx_s\leq \idx_t)$ mapping $\beta_{\idx_s}^{pq}$ bijectively into $\beta_{\idx_t}^{pq}$. 
    % Therefore, it follows that $f(a\leq \idx_t)$ maps $\beta_{a}^{pq}$ bijectively into $\beta_{\idx_t}^{pq}$.} 
    %
    % \textcolor{blue}{For values of $p$ and $q$ such that $\idx_s<\idx_p\leq a$ and $\idx_q>\idx_t$, recall that we constructed the basis $\beta_a^{pq}$ of $A_a^{pq}$ by taking the basis $\beta_{\idx_t}^{pq}$ of $A_{\idx_t}^{pq}$ and choosing a preimage for each element. 
    % This naturally implies that $f(a\leq \idx_t)$ maps $\beta_{a}^{pq}$ bijectively into $\beta_{\idx_t}^{pq}$.}
    %   
    % \textcolor{blue}{Finally, for values of $p$ and $q$ such that $\idx_q\leq \idx_t$, note that $A^{pq}_a\subseteq \Ker[a,\idx_t]$. 
    % Therefore, $f(a\leq \idx_t)$ maps the corresponding basis $\beta_{a}^{pq}$ to $0$. }
    %
    % \textcolor{blue}{Therefore, we have that $f(a\leq \idx_t)$ maps $\beta_a\setminus \Ker[a,\idx_t]$ injectively into $\beta_{\idx_t}$, which implies $[f(a\leq \idx_t)]_{\beta_a}^{\beta_{\idx_t}}$ is a matching matrix.}
\end{proof}

We can extend Lemma \ref{lemma:matching_matrix} by the following corollary.

\begin{theorem}
    [Matching matrices for $(S\cap J)\cup\{a\}$]
    \label{thm:matching_matrix}
    For any $\idx_p\in S\cap J$ such that $\idx_p<a$, the matrix representation $[f(\idx_p\leq a)(A)]_{\beta_{\idx_p}}^{\beta_a}$ with respect to bases $\beta_{\idx_p}$ and $\beta_a$ is a matching matrix. 
    Similarly, for any $\idx_q\in S\cap J$ such that $\idx_q>a$, the matrix representation $[f(a\leq \idx_q)(A)]_{\beta_a}^{\beta_{\idx_q}}$ with respect to bases $\beta_a$ and $\beta_{\idx_q}$ is a matching matrix.
\end{theorem}

\begin{proof}
    For any $\idx_p\in S\cap J$ such that $\idx_p<a$, observe that $f(\idx_p\leq a) = f(\idx_s\leq a)\circ f(\idx_p\leq \idx_s)$. 
    So, $[f(\idx_p\leq a)]_{\beta_{\idx_p}}^{\beta_a} = [f(\idx_s\leq a)]_{\beta_{\idx_s}}^{\beta_a} [f(\idx_p\leq \idx_s)]_{\beta_{\idx_p}}^{\beta_{\idx_s}}$. 
    By Lemma \ref{lemma:matching_matrix}, the matrix representation $[f(\idx_s\leq a)]_{\beta_{\idx_s}}^{\beta_a}$ is a matching matrix, and $[f(\idx_p\leq \idx_s)]_{\beta_{\idx_p}}^{\beta_{\idx_s}}$ is a matching matrix by assumption that $\{\beta_j\}_{j\in J}$ is a consistent basis with respect to $f$.
    Therefore, it follows that $[f(\idx_p\leq a)]_{\beta_{\idx_p}}^{\beta_a}$ is a matching matrix. 

    Similarly, for any $\idx_q\in S\cap J$ such that $\idx_q>a$, observe that $f(a\leq \idx_q) = f(\idx_t\leq \idx_q)\circ f(a\leq \idx_t)$. 
    By a similar argument, $[f(a\leq \idx_q)]_{\beta_{a}}^{\beta_{\idx_q}}$ is a matching matrix.
    % So, $[f(a\leq \idx_q)]_{\beta_{a}}^{\beta_{\idx_q}} = [f(\idx_t\leq \idx_q)]_{\beta_{\idx_t}}^{\beta_{\idx_q}} [f(a\leq \idx_t)]_{\beta_{a}}^{\beta_{\idx_t}}$. 
    % By Lemma \ref{lemma:matching_matrix}, the matrix representation $[f(a\leq \idx_t)]_{\beta_{a}}^{\beta_{\idx_t}}$ is a matching matrix, and $[f(\idx_t\leq \idx_q)]_{\beta_{\idx_t}}^{\beta_{\idx_q}}$ is a matching matrix by assumption that $\{\beta_j\}_{j\in J}$ is a consistent basis with respect to $f$.
    % Therefore, it follows that $[f(a\leq \idx_q)]_{\beta_{a}}^{\beta_{\idx_q}}$ is a matching matrix. 
\end{proof}

We now show $(\beta_j)_{j\in J\cup\{a\}}$ forms a consistent basis for $f|_{J\cup\{a\}}$ with respect to $f$. 

\begin{theorem}    
    \label{thm:consistent_basis}
    The family of bases $(\beta_j)_{j\in J\cup\{a\}}$ forms a consistent basis for $f|_{J\cup\{a\}}$ with respect to $f$. 
\end{theorem}

We show Theorem \ref{thm:consistent_basis} in two steps. 
We first show, in Theorem \ref{thm:matching_matrix_general}, that the matrix representation of every structure map with respect to the bases $(\beta_j)_{j\in J\cup\{a\}}$ is a matching matrix.
Next, we show, in Theorem \ref{theorem:image/kernel_containment}, that every $\beta_j$, where $j\in J\cup\{a\}$, contains a basis for all of the kernel and image submodules that are required for a basis to meet the criteria of a consistent basis with respect to $f$.

\begin{figure}[h]
    \centering
    \includegraphics[width=0.7\textwidth]{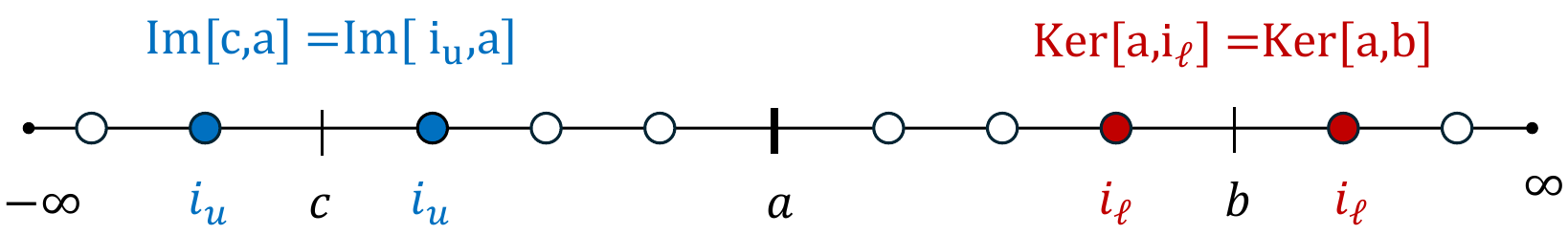}
    \caption{Indices in the proof of Theorem \ref{thm:matching_matrix_general}.
    % ; circles with black borders on the x-axis represent elements of $S$. 
    % In this argument we choose $c < a < b$ such that $c, b \in J$. 
    % We also choose
    For $c,b\in J$ such that $c < a < b$, we choose $\idx_u,\idx_\ell \in S$ such that $\mathrm{Im}[c,a] = \mathrm{Im}[\idx_u,a]$ and $\ker[a,b] = \ker[a, \idx_\ell]$.
    The index $\idx_u$ (respectively, $\idx_\ell$) may be greater or less than $c$ (respectively, $b$).
    % ; the element $\idx_u$ may be greater or less than $c$. We also choose $\idx_\ell \in S$ such that $\ker[a,b] = \ker[a, \idx_\ell]$; the element $\idx_\ell$ may be greater or less than $c$.
    }
    \label{fig:matching_basis_for_j_proof}  % Optional: for cross-referencing
\end{figure}

\begin{theorem}
    [Matching matrices for $J\cup\{a\}$]
    \label{thm:matching_matrix_general}
    The matrix representation of every structure map with respect to the bases $(\beta_j)_{j\in J\cup\{a\}}$ is a matching matrix.
    That is, for every $c, b\in J\cup\{a\}$ \jerry{such that $c \le b$}, the matrix representation $[f(c\leq b)]_{\beta_c}^{\beta_b}$ is a matching matrix.
\end{theorem}

\begin{proof}
This proof involves several elements of $I$, labeled $a,b,c,\idx_u,$ and $\idx_\ell$. See Figure \ref{fig:matching_basis_for_j_proof} for a visual aid of the relative positions of these indices.

It is enough to focus only on structure maps that start or end at $f_a$, because $(\beta_j)_{j\in J}$ (i.e., without $a$) is assumed to be a consistent basis with respect to $f$.

% \greg{[[Comment: changing $\idx_x$ to $\idx_c$ to avoid notation conflict.]]}\jerry{[i changed $c\mapsto u$ because technically $c$ is in use]}
We first look at maps that end at $f_a$. 
Take any $c\in J$ such that $c<a$. 
Choose $\idx_u\in S\cap J$ such that $\im[\idx_u, a] = \im[c,a]$. 
If $c\leq \idx_u \leq a$, then $[f(c\leq a)]_{\beta_c}^{\beta_a} = [f(\idx_u\leq a)]_{\beta_{\idx_u}}^{\beta_a} [f(c\leq \idx_u)]_{\beta_c}^{\beta_{\idx_u}}$.
Note that $[f(c\leq \idx_u)]_{\beta_c}^{\beta_{\idx_u}}$ is a matching matrix because $\idx_u,c\in J$, and $(\beta_j)_{j\in J}$ is a consistent basis.
% with respect to $f$. 
Additionally, by Theorem \ref{thm:matching_matrix}, $[f(\idx_u\leq a)]_{\beta_{\idx_u}}^{\beta_{a}}$ is a matching matrix. 
Therefore $[f(c\leq a)]_{\beta_c}^{\beta_a}$ is a matching matrix.

Now, consider $\idx_u\leq c < a$, and let  $\{v_1,\ldots,v_r\} := \beta_{\idx_u}\setminus \Ker[\idx_u,a]$. 
By Theorem \ref{thm:matching_matrix}, this set maps injectively into $\beta_a$, and its image under   $f(\idx_u\leq a)$ spans $\im[\idx_u,a] = \im[c,a]$.
% This implies that $f(\idx_u\leq c)(v_1),\ldots, f(\idx_u\leq c)(v_r)$, which are elements in $\beta_{c}$, map injectively into $\beta_a$, and their image under $f(c\leq a)$ maps injectively into $\beta_a \subseteq f_a$. 
By functoriality, this implies that the image of the set $\{f(\idx_u\leq c)(v_1),\ldots, f(\idx_u\leq c)(v_r)\} \subseteq \beta_{c}$ maps injectively into the basis $\beta_a \subseteq f_a$ under $f(c\leq a)$.
Because $\beta_c$ contains a basis for $\Ker[c,a]$, we have, by the rank-nullity theorem, that $\beta_c\setminus\{f(\idx_u\leq c)(v_1),\ldots, f(\idx_u\leq c)(v_r)\}$ spans $\Ker[c,a]$. 
Therefore, it follows that $[f(c\leq a)]_{\beta_c}^{\beta_a}$ is a matching matrix.

% \greg{[[Comment: changing $\idx_t$ to $\idx_b$ to deconflict notation.]]}
% \jerry{[see above]}
Next, we look at maps that begin at $f_a$.
Take any $b\in J$ such that $b>a$. 
By definition, there exists $\idx_\ell\in S\cap J$ such that $\Ker[a,\idx_\ell]=\Ker[a,b]$.
If $a<\idx_\ell\leq b$, then $[f(a\leq b)]_{\beta_a}^{\beta_b} = [f(\idx_\ell\leq b)]_{\beta_{\idx_\ell}}^{\beta_b}[f(a\leq \idx_\ell)]_{\beta_a}^{\beta_{\idx_\ell}}$ is a matching matrix because both $[f(\idx_\ell\leq b)]_{\beta_{\idx_\ell}}^{\beta_b}$ and $[f(a\leq \idx_\ell)]_{\beta_a}^{\beta_{\idx_\ell}}$ (by Theorem \ref{thm:matching_matrix}) are matching matrices.
% , so too is $[f(a\leq b)]_{\beta_a}^{\beta_b}$. 

Now, suppose that $a<b\leq \idx_\ell$.
Let $v\in \beta_a\setminus \Ker[a,\idx_\ell]$ be given. 
Because $\Ker[a,\idx_\ell] = \Ker[a,b]$, the elements $f(a\le b)(v)$ and ${f(a \le i_\ell)(v)}$
% $$
% \underbrace{f(a \le i_\ell)(v)}_{\in \beta_{\idx_\ell}} 
% = 
% f(b \le \idx_\ell)\Bigl(  f(a\le b)(v) \Bigr)
% $$
are nonzero.
We know that $f(a\leq \idx_\ell)(v)\in \beta_{\idx_\ell}$, by Theorem \ref{thm:matching_matrix}. Therefore, because $(\beta_j)_{j\in J}$ is a consistent basis,
% with respect to $f$, 
there exists $w\in \beta_b$ such that $f(b\leq \idx_\ell)(w)=f(a\leq \idx_\ell)(v)$. 

We will show that $f(a\leq b)(v)=w$. 
To see this, note that because $\Ker[a,b]=\Ker[a,\idx_\ell]$, functoriality implies that $\im[a,b]\cap \Ker[b,\idx_\ell]=0$. 
Let $\{w_1,\ldots,w_r\} \subseteq \beta_b$ be a basis for $\im[a,b]$, in which case, we can write $f(a\leq b)(v)=c_1w_1+\ldots+ c_rw_r$. 
By applying $f(b\leq \idx_\ell)$ to both sides, we get $f(a\leq \idx_\ell)(v)=c_1f(b\leq \idx_\ell)(w_1)+\ldots+c_rf(b\leq \idx_\ell)(w_r)$.
Because $\im[a,b]\cap\Ker[b,\idx_{\ell}]=0$ and $(\beta_j)_{j\in J}$ is a consistent basis with respect to $f$,
% $\Ker[a,b]=\Ker[a,\idx_\ell]$, 
we must have $f(b\leq \idx_\ell)(w_1),\ldots,f(b\leq \idx_\ell)(w_r)\in \beta_{\idx_\ell}$, and because we also have $f(a\leq \idx_\ell)(v)\in\beta_{\idx_\ell}$, it follows that all but one of the $c_i$ are $0$, with the remaining one being $1$. 
Without loss of generality, suppose $c_1=1$ and $c_2=\ldots=c_r=0$, in which case, we have $f(a\leq \idx_\ell)(v)=f(b\leq \idx_\ell)(w_1)$; because $w\in\beta_b$ is the unique basis element in $\beta_{b}$ that maps to $f(a\leq \idx_\ell)(v)$, it follows that $w=w_1$.
This yields $f(a\leq b)(v) = w$, as desired.
% Therefore, we have 
% \begin{align*}
% 	f(a\leq b)(v)&=c_1w_1+\cdots c_rw_r\\
% 	&= w_1 \\
% 	&= w
% \end{align*}
% as desired. 

From this, it follows that $f(a\leq b)$ maps $\beta_a\setminus\Ker[a,\idx_\ell]=\beta_a\setminus\Ker[a,b]$  into $\beta_b$. 
Moreover, functoriality implies that this map is injective, because the composite map $f(a \le \idx_\ell) = f(b \le \idx_\ell) \circ f(a \le b)$ is injective on $\beta_a\setminus\Ker[a,\idx_\ell]=\beta_a\setminus\Ker[a,b]$, by Theorem \ref{thm:matching_matrix}.
This implies that $[f(a\leq b)]_{\beta_a}^{\beta_b}$ is a matching matrix, as desired.
\end{proof}

Next, we show that for every $j\in J\cup\{a\}$, the basis $\beta_j$ contains a basis for every $\Ker[j,b]$ and every $\im[c,j]$. 
We know that for $j\in J$, the basis $\beta_j$ contains a basis for every $\Ker[j,b]$ and every $\im[c,j]$ because the indexed family $(\beta_j)_{j\in J}$ (without the basis $\beta_a$ appended) is assumed to be a consistent basis with respect to $f$ (see Definition \ref{def:bases}).
Therefore, we only need to show this for $j=a$.

We use the following fact.

\begin{lemma}\label{lemma:consistent basis_im_ker}
    Suppose that $f: I \to R\mathrm{-Mod}$ is any persistence module with consistent basis $(\beta_i)_{i \in I}$.
    Then, for any $a\in I$, the basis $\beta_a$ contains a basis for every submodule of the form $\im[c,a]$  and $\Ker[a,b]$.
\end{lemma}
\begin{proof}
    % This may be shown using the rank-nullity theorem.
    For any $x,y\in I$ such that $x\leq y$, define $\mathcal{K}_{x,y} = \beta_x \cap \Ker[x,y]$, and $\mathcal{I}_{x,y} = \beta_x\setminus \mathcal{K}_{x,y}$. 
    Let $c\leq a \le b$ be given.
    Clearly, $f(c\leq a)(\mathcal{I}_{c,a})$ spans the image submodule $\im[c,a]$. 
    % Further, $L$ is linearly independent because $L \subseteq \gamma_a$.
    % Therefore $\mathrm{rank}(\gamma(c \le a)) = |L| = |\mathcal{I}_{c,a}$. 
    % We have that $\mathcal{K}_{a,b}\subseteq \gamma_a$ is a basis for $\Ker[a,b]$.
    Because $[f(c\leq a)]_{\beta_c}^{\beta_a}$ is a matching matrix, the set $\mathcal{I}_{c,a}$ maps injectively into $\beta_a$, and therefore, $f(c\leq a)(\mathcal{I}_{c,a})\subseteq \beta_a$ forms a basis for $\im[c,a]$.
    Moreover, by the rank nullity theorem, $\mathcal{K}_{a,b}$ forms a basis for $\Ker[a,b]$. 
    % Since $\gamma_a = \mathcal{K} \sqcup \mathcal{I}$, and $g(p \le q)$ maps $I$ injectively into a linearly independent set.
\end{proof}

% \begin{lemma}
%     Let $h: A \to B$ be a homomorphism of free and finitely generated modules over $R$. Suppose that $\beta_A$ and $\beta_B$ are bases for $A$ and $B$ respectively, and suppose that the matrix representation of $h$ with respect to $\beta_A$ and $\beta_B$ is a matching matrix. Then $\beta_A$ contains a basis for $\ker(h)$ and $\beta_B$ contains a basis for $\mathrm{Im}(h)$.
% \end{lemma}
% \begin{proof}
%     Let $M$ be the matrix representation of $h$ with respect to the given bases. Denote by $M[:,a]$ the column of the $M$ corresponding to an element $a \in \beta_A$ and denote by $M[b,:]$ the row of the $M$ corresponding to an element $b \in \beta_B$. Then because $M$ is a matching matrix, an elementary exercise shows that $\{ a \in \beta_A : M[:,a] = 0\}$ is a basis for $\ker(h)$ and $\{b \in \beta_B : M[b,:] \neq 0 \}$ is a basis for $\mathrm{Im}(h)$.
% \end{proof}

\begin{theorem}\label{theorem:image/kernel_containment}
    The basis $\beta_a$ for $f_a$ contains a basis for every $\Ker[a,b]$ and $\im[c,a]$. 
\end{theorem}

\begin{proof}
Theorem \ref{thm:matching_matrix} already established that $(\beta_j)_{j \in S}$ is a consistent basis for $f|_S$.  Therefore, by Lemma \ref{lemma:consistent basis_im_ker}, the set $\beta_a$ contains a basis for $\im[s,a]$ and $\Ker[a,s]$ for all $s \in S$. 
By hypothesis, every submodule of the form $\Ker[a,b]$ (respectively, $\im[c,a]$) can be expressed in the form $\Ker[a,s]$ (respectively, $\im[s,a]$) for some $s \in S$. 
The desired conclusion follows.
%
% Let $c,b\in J$ such that $c<a<b$.%
%
% % \greg{[[Comment: changing $\idx_t$ to $\idx_y$ to deconflict notation.]]}
% % \jerry{[i changed it to $\idx_\ell$ because $y$ is technically in use.]}
% We can find $\idx_\ell\in S$ such that $\Ker[a,b]=\Ker[a,\idx_\ell]$.
% Recall that $\beta_a = \bigsqcup_{p,q}\beta_a^{pq}$, where each $\beta_a^{pq}$ is a basis of $A_a^{pq}$, which is a complement of $\IK_a[\idx_{p-1},q]+\IK_a[\idx_p,\idx_{q-1}]$ in $\IK_a[\idx_p,\idx_q]$.
% By \cite[Theorem 18]{Luo2023}, we have that 
% \begin{equation}
% \Ker[a,\idx_\ell]=\bigoplus_{q\leq \ell} A^{pq}_a \,,
% \end{equation}
% and therefore, 
% \begin{equation}
% \bigcup_{q\leq \ell} \beta^{pq}_a \subseteq \beta_a
% \end{equation}
% forms a basis for $\Ker[a,b]$. 
%
% % \greg{[[Comment: changing $\idx_{\greg{x}}$ to $\idx_x$ to deconflict notation.]]}
% % \jerry{[Ditto here; $x\mapsto u$]}
% Similarly, we can find $\idx_u\in S$ such that $\im[c,a]=\im[\idx_u,a]$,
% and we can write 
% \begin{equation}
% \im[c,a]=\bigoplus_{p\leq u} A^{pq}_a\,,
% \end{equation}
% and therefore, 
% \begin{equation}
% \bigcup_{p\leq u} \beta^{pq}_a \subseteq \beta_a
% \end{equation} 
% forms a basis for $\im[c,a]$. 
% Therefore, $\beta_a$ contains a basis for every $\Ker[a,b]$ and every $\im[c,a]$.
\end{proof}

% \greg{[[[We can delete the following proof and the sentence that precedes it (i added one higher up, directly under Theorem \ref{thm:extension}]]]---------------------}

% Theorems \ref{thm:matching_matrix_general} and \ref{theorem:image/kernel_containment} yield the proof of Theorem \ref{thm:extension}, which we briefly summarize.

% \begin{proof}[Proof of Theorem \ref{thm:extension}]
% In the preceding discussion we fixed a subset $J$ such that $\{-\infty, \infty\} \subseteq J \subsetneq I$; chose an arbitrary $a\in I\setminus J$; and constructed a basis $\beta_a\subseteq f_a$. 
% The family of bases $(\beta_j)_{j\in J\cup\{a\}}$ is a consistent basis with respect to $f$ by Theorem \ref{thm:consistent_basis}, as desired. 
% % Consider the family of bases $(\beta_j)_{j\in J\cup\{a\}}$. 
% % \greg{Theorem \ref{thm:matching_matrix_general} implies that the matrix representation of every structure map with respect to $(\beta_j)_{j\in J\cup\{a\}}$ is a matching matrix. }
% % Moreoever, by Theorem \ref{theorem:image/kernel_containment}, $\beta_a$ contains a basis for every $\Ker[a,y]$ and every $\im[x,a]$. 
% % Recall that because $(\beta_j)_{j\in J}$ is a consistent basis with respect to $f$, for every $j\in J$, we have that $\beta_j$ contains a basis for every $\Ker[j,y]$ and every $\im[x,j]$. 
% % Therefore, it follows that $(\beta_j)_{j\in J\cup\{a\}}$ is a consistent basis with respect to $f$. 
% \end{proof}

\section{\textbf{Application: Integer decomposition, relative homology, and field-choice independence in persistent homology}}
\label{sec:applications}

% As an application of our main result, we give a precise criteria for decomposability of integer persistent homology, and for the invariance of persistence diagrams to the choice of coefficient field. 
{We relate the decomposability of integer persistent homology to the sensitivity of persistence diagrams to choice of coefficient field.}
% As an application of our main result, we relate the decomposability of integer persistent homology to relative homology and the sensitivity of persistence diagrams to the choice of coefficient field.
{Let $\mathcal{K}$ be a filtration, viewed as a functor from $I$ to the category of topological spaces and inclusions.}
% Let $\mathcal{K}$ be a filtration, viewed as a functor from $I$ to the category of topological spaces and inclusions between them. 
{Let 
% $\mathcal{H}_n(\mathcal{K}; F)$ 
\jerry{${H}_n(\mathcal{K}; F)$}
be the functor obtained by composing $\mathcal{K}$ with the $n$th-homology functor over a field $F$.} 
% Let $\mathcal{H}_n(\mathcal{K}; F)$ be the \emph{$n$th persistent homology functor} --- which is obtained by composing $\mathcal{K}$ with the $n$th-homology functor --- with coefficients in a field $F$. 
The \emph{persistence diagram} of $\mathcal{K}$ with coefficients in $F$, denoted $\mathrm{PD}_n^F(\mathcal{K})$, is the multiset of intervals corresponding to an interval decomposition of 
\jerry{${H}_n(\mathcal{K}; F)$.}
% $\mathcal{H}_n(\mathcal{K}; F)$.
% The \emph{persistence diagram} of $\mathcal{K}$ with coefficients in $F$ is the multiset of intervals corresponding to an interval decomposition of $\mathcal{H}_n(\mathcal{K}; F)$, denoted $\mathrm{PD}_n^F(\mathcal{K})$.
% this quantity is well-defined because $F$ is a field.  
The \emph{field-independence problem} asks: under what conditions is the persistence diagram $\mathrm{PD}_n^F(\mathcal{K})$ independent of the choice of $F$?  
Obayashi and Yoshiwaki \cite[\jerry{Theorem 1.6}]{obayashi_field_2023} 
\jerry{proved that, when $I$ is finite, a persistence diagram of a filtration is independent of field choice if and only if the relative homology group $H_n(\mathcal{K}_p, \mathcal{K}_q; \mathbb{Z})$ is free for all $q \le p$, and $H_{n-1}(\mathcal{K}_p;\mathbb{Z})$ is free for all $p$.}
% addressed this in the special case where $I$ is finite. 
We extend this to the infinite case. 

\begin{theorem}
% [\greg{See Definition \ref{def:bases} and Section \ref{sec:background} for definitions}]
\label{thm:OurFieldIndependence}
    Suppose that the homology groups $H_{n-1}(\mathcal{K}_a;\Z)$ and $H_{n}(\mathcal{K}_a;\Z)$  are free and finitely-generated for all $a \in I$. 
    % The following diagram commutes for all $a \le b$, where maps marked $(*)$ are extension of scalars and maps marked $(\dagger)$ are isomorphisms provided by the universal coefficient theorem.
    Given any $a\leq b$, consider the following commutative diagram, where $(*)$ are extension of scalars and $(\dagger)$ arise from the universal coefficient theorem.
    \begin{equation}
    \label{eq:tensor_cd}
    \begin{tikzcd}
	{H_n(\mathcal{K}_a;\mathbb{Z})} & {H_n(\mathcal{K}_a;\mathbb{Z}) \otimes F} & {H_n(\mathcal{K}_a;F)} \\
	{H_n(\mathcal{K}_b;\mathbb{Z})} & {H_n(\mathcal{K}_b;\mathbb{Z})\otimes F} & {H_n(\mathcal{K}_b;{F})}
	\arrow["{(*)}", from=1-1, to=1-2]
	\arrow[from=1-1, to=2-1]
	\arrow["{(\dagger)}", "{\cong}"', from=1-2, to=1-3]
    \arrow[from=1-2, to=2-2]
	\arrow[from=1-3, to=2-3]
	\arrow["{(*)}", from=2-1, to=2-2]
	\arrow["{(\dagger)}", "{\cong}"', from=2-2, to=2-3]
    \end{tikzcd}
    \end{equation}

    % \greg{The maps $H_{n}(\mathcal{K}_a;\Z) \to H_n(\mathcal{K}_a;F)$ therefore yield a homomorphism of persistence modules (i.e. natural transformation) $\psi: \mathcal{H}_n(\mathcal{K}; \Z) \to \mathcal{H}_n(\mathcal{K}; F)$. This map  preserves direct sum decompositions, and therefore consistent bases.}
    % $$
    % \mathcal{H}_n(\mathcal{K}; \Z) \to \mathcal{H}_n(\mathcal{K}; F).
    % $$ 
    \noindent The maps $(\dagger)$ are isomorphisms, and the following are equivalent.
    \begin{enumerate}
        \item The  module $H_n(\mathcal{K};\Z)$ splits as a direct sum of interval submodules.    
        \item The cokernel of each induced  map $H_n(\mathcal{K}_a;\Z) \to H_n(\mathcal{K}_b;\Z)$  is free.
        \item For all $a \le b$ in $I$, the relative homology group $H_n(\mathcal{K}_b,\mathcal{K}_a;\Z)$ is free. 
        \item The persistence diagram $\textrm{PD}_n^F(\mathcal{K})$ is identical for every coefficient field $F$.   
    \end{enumerate}
    
\end{theorem}
\begin{proof}
    If $\textrm{PD}_n^F(\mathcal{K})$ is identical for every coefficient field $F$, then the same holds for $\mathcal{K}|_S$, for every finite subset $S \subseteq I$. It follows by \cite[Theorem 1.9]{obayashi_field_2023} that (4) implies (3). The argument from \cite[Theorem 4]{Luo2023} carries over without modification to show that (3) implies (2). Our main result, Theorem \ref{thm:problem_statement}, shows that (2) implies (1).

    {Because $H_{n-1}(\mathcal{K}_a;\Z)$ and $H_{n}(\mathcal{K}_a;\Z)$ are free for all $a$, the maps $(\dagger)$ in \eqref{eq:tensor_cd} are isomorphisms, by the universal coefficient theorem. 
    The extension maps $(*)$ induce homomorphisms $\varphi: H_n(\mathcal{K};\mathbb{Z}) \to H_n(\mathcal{K};\mathbb{Z}) \otimes F$ of persistence modules, which preserves direct sums, in the sense that $\varphi(\bigoplus_i A_i) = \bigoplus_i\varphi(A_i)$. 
    Therefore, the same holds for the map $\psi: H_n(\mathcal{K};\mathbb{Z}) \to H_n(\mathcal{K};F) $ obtained by composing $\varphi$ with the isomorphism $H_n(\mathcal{K};\mathbb{Z}) \otimes F \xrightarrow{\cong} H_n(\mathcal{K};F)$.}  
    Consequently, if $H_n(\mathcal{K}; \Z)$ splits as a direct sum of interval submodules, then $\psi$ yields a corresponding decomposition of $H_n(\mathcal{K}; F)$, with the same multiset of intervals. 
    This establishes that (1) implies (4). 
    The same homomorphism carries consistent bases to consistent bases.
    %
    % Finally, the fact that $\mathcal{H}_n(\mathcal{K}; \Z) \to \mathcal{H}_n(\mathcal{K}; F)$ carries consistent bases to consistent bases follows from the fact that the maps $H_{n}(\mathcal{K}_a;\Z) \to H_{n}(\mathcal{K}_a;\Z) \otimes F \xrightarrow{\cong} H_{n}(\mathcal{K}_a;F)$ both preserve bases.
    % The proof is essentially that of Theorem 4 in \cite{Luo2023}. The only adjustments required are (i) to invoke our main result, Theorem \ref{thm:problem_statement}, as a generalization of \cite{Luo2023}[Theorem 1], and (ii) to apply the argument in \cite{Luo2023}[Section 3.3] to establish equivalence between 
\end{proof}

\begin{remark}
    An analogous result holds for \textbf{persistent cohomology}.
\end{remark}

\appendix 

\section{\textbf{A useful short exact sequence}}
\label{appendix:ses}

Here we present a short exact sequence used to prove Lemma \ref{lemma:valid_A^pq}. 
The notation is the same as in the proof of that result.

\begin{lemma}\label{lemma:short_exact_sequence}
    Fix $\idx_s<\idx_p\leq a<\idx_t<\idx_q$, and consider the diagram 
\[\begin{tikzcd}[column sep=scriptsize, row sep=scriptsize, font=\small, cramped]
0 & {\IK_a[\idx_{p-1},\idx_q] + \IK_a[\idx_p,\idx_{q-1}]} & {\IK_a[\idx_p,\idx_q]} & {A^{pq}_{\idx_t}} & 0 \\
&&& {\IK_{\idx_t}[\idx_p,\idx_q]}
\arrow[from=1-1, to=1-2]
\arrow["\iota", hook, from=1-2, to=1-3]
\arrow["\phi", dashed, two heads, from=1-3, to=1-4]
\arrow["f(a\leq \idx_t)|_{\IK_a[\idx_p,\idx_q]}"', two heads, from=1-3, to=2-4]
\arrow["\ell"', dashed, bend right=45, from=1-4, to=1-3]
\arrow[from=1-4, to=1-5]
\arrow["\pi"', two heads, from=2-4, to=1-4]
\end{tikzcd}\]
% consisting of $\iota$ and $\phi$ is exact, where
% $\iota$ is the inclusion map;
% $\pi$ is the projection map  which restricts to zero on $\IK_{\idx_t}[\idx_{p-1},\idx_q]+\IK_{\idx_t}[\idx_{p},\idx_{q-1}]$ and to the identity map on $A^{pq}_{\idx_t}$; and
% $\phi = \pi\circ f(a\leq \idx_t)|_{\IK_a[\idx_p,\idx_q]}$.
where
$\iota$ is the inclusion map;
$\pi$ is the projection map  which restricts to zero on $\IK_{\idx_t}[\idx_{p-1},\idx_q]+\IK_{\idx_t}[\idx_{p},\idx_{q-1}]$ and to the identity map on $A^{pq}_{\idx_t}$; and
$\phi = \pi\circ f(a\leq \idx_t)|_{\IK_a[\idx_p,\idx_q]}$. Then the sequence consisting of $\iota$ and $\phi$ is exact.
% \begin{itemize}
%     \item  $\iota$ is the inclusion map;
%     \item $\pi$ is the projection map  which restricts to zero on $\IK_{\idx_t}[\idx_{p-1},\idx_q]+\IK_{\idx_t}[\idx_{p},\idx_{q-1}]$ and to the identity map on $A^{pq}_{\idx_t}$;
%     \item $\phi = \pi\circ f(a\leq \idx_t)|_{\IK_a[\idx_p,\idx_q]}$.   
% \end{itemize}
Moreover, this sequence splits. 
\end{lemma}

\begin{proof}
    % The sequence is exact at $\IK_a[\idx_{p-1},\idx_q] + \IK_a[\idx_p,\idx_{q-1}]$ because the inclusion map $\iota$ is injective. 
    {The sequence is exact at $\IK_a[\idx_{p-1},\idx_q] + \IK_a[\idx_p,\idx_{q-1}]$ because $\iota$ is injective.} 
    
    {To check exactness at $A_{\idx_t}^{pq}$, because $\pi$ is surjective, it suffices show that $f(a\leq i_t)$ restricts to a surjection $\IK_a[\idx_p,\idx_q] \to \IK_{\idx_t}[\idx_p,\idx_q]$.}
    % To check exactness at $A_{\idx_t}^{pq}$, it suffices show that $f(a\leq i_t)$ restricts to a surjection  $\IK_a[\idx_p,\idx_q] \to \IK_{\idx_t}[\idx_p,\idx_q]$ (indeed, $\pi$ is surjective as it is a projection).
    By \cite[Theorem 15]{Luo2023}, we have $f(a\leq \idx_t)(\IK_a[\idx_p,\idx_q]) = \IK_{\idx_t}[\idx_p,\idx_q]\cap \im[a,\idx_t]$.
    However, because $\idx_p\leq a$, we have $\im[\idx_p,\idx_t]\subseteq \im[a,\idx_t]$, which implies that
    % \jerry{$\IK_{\idx_t}[\idx_p,\idx_q]\cap \im[a,\idx_t] = \IK_{\idx_t}[\idx_p,\idx_q]$.}
    \begin{align*}
        \IK_{\idx_t}[\idx_p,\idx_q]\cap \im[a,\idx_t] &= \im[\idx_p,\idx_t]\cap\Ker[\idx_t,\idx_q]\cap \im[a,\idx_t]\\
        &= \im[\idx_p,\idx_t]\cap\Ker[\idx_t,\idx_q]\\
        &= \IK_{\idx_t}[\idx_p,\idx_q]\,{.}
    \end{align*}
    % as desired. 

    % Finally, we show exactness at $\IK_a[\idx_p,\idx_q]$. 
    % To show this, it is enough to show that 
    {
    Finally, we show exactness at $\IK_a[\idx_p,\idx_q]$. It is enough to show that} 
    \begin{align}
    \Bigl (f(a\leq \idx_t)|_{\IK_a[\idx_p,\idx_p]} \Bigr)^{-1} \Bigl(\IK_{\idx_t}[\idx_{p-1},\idx_q] &+ \IK_{\idx_t}[\idx_p,\idx_{q-1}] \Bigr) \label{eq:sufficient_equality_ses}\\
    &= \IK_a[\idx_{p-1},\idx_q] + \IK_a[\idx_p,\idx_{q-1}]\,.\nonumber
    \end{align}
    Note that by \cite[Theorem 16]{Luo2023}, we have 
    % \begin{align}
    %     f(a\leq {\idx_t})^{-1}\Bigl(\IK_{\idx_t}[\idx_{p-1},\idx_q] + \IK_{\idx_t}[\idx_p,\idx_{q-1}] \Bigr) &= \IK_{{a}}[\idx_{p-1},\idx_q] + \IK_{{a}}[\idx_p,\idx_{q-1}] \nonumber\\
    %     &+\Ker[a,\idx_t]\,.\nonumber
    % \end{align}
    \[
        f(a\leq {\idx_t})^{-1}\Bigl(\IK_{\idx_t}[\idx_{p-1},\idx_q] + \IK_{\idx_t}[\idx_p,\idx_{q-1}] \Bigr) = \IK_{{a}}[\idx_{p-1},\idx_q] + \IK_{{a}}[\idx_p,\idx_{q-1}] +\Ker[a,\idx_t]\,.\]
    Therefore, it follows that 
        \begin{align}
        \Bigl(f(a\leq {\idx_t})|_{\IK_{a}[\idx_p,\idx_q]} \Bigr)^{-1}&\Bigl(\IK_{\idx_t}[\idx_{p-1},\idx_q] + \IK_{\idx_t}[\idx_p,\idx_{q-1}] \Bigr)\nonumber \\
        &= \Bigl(\IK_{{a}}[\idx_{p-1},\idx_q] + \IK_{{a}}[\idx_p,\idx_{q-1}]
        +\Ker[a,\idx_t]\Bigr)\cap \IK_{a}[\idx_p,\idx_q]\label{eq:first_ses}\\
        &=\Bigl(\IK_{{a}}[\idx_{p-1},\idx_q] + \IK_{{a}}[\idx_p,\idx_{q-1}]\Bigr)+ \Bigl ( \Ker[a,\idx_t]\cap \IK_{a}[\idx_p,\idx_q] \Bigr )\label{eq:second_ses}\\
        &= \Bigl (\IK_{{a}}[\idx_{p-1},\idx_q] + \IK_{{a}}[\idx_p,\idx_{q-1}] \Bigr) + \IK_{a}[\idx_p,\idx_t]\label{eq:third_ses}\\
        &= \IK_{{a}}[\idx_{p-1},\idx_q] + \IK_{{a}}[\idx_p,\idx_{q-1}]\,.\label{eq:fourth_ses}
    \end{align}
    
    The equality between \eqref{eq:first_ses} and \eqref{eq:second_ses} follows from the modularity of submodules\footnote{{An order lattice is \emph{modular} if, for any $a \le b$, the identity $(a \vee x) \wedge b = a \vee (x \wedge b)$ holds for all $x$. The lattice of submodules of a module over a ring, ordered under inclusion, is always modular. In this case, we take 
        $a  = \IK_{{a}}[\idx_{p-1}, \idx_q] + \IK_{{a}}[\idx_p,\idx_{q-1}]$,
        $b  = \IK_{{a}}[\idx_{p},\idx_q]$, and 
        $x  = \Ker[a,\idx_{{t}}]$.
    % \begin{align*}
    %     a & = \IK_{{a}}[\idx_{p-1},\idx_q] + \IK_{{a}}[\idx_p,\idx_{q-1}]\; \\
    %     b & = \IK_{{a}}[\idx_{p},\idx_q]\; \\
    %     x & = \Ker[a,\idx_t]\,.   
    % \end{align*}
    }}, and the equality between \eqref{eq:second_ses} and \eqref{eq:third_ses} holds because $\idx_t < \idx_q$ implies that $\Ker[a,\idx_t] \cap \IK_a[\idx_p, \idx_q] = \Ker[a,\idx_t] \cap \Ker[a, \idx_q]  \cap \mathrm{Im}[\idx_p, a] = \Ker[a, \idx_t] \cap \mathrm{Im}[\idx_p, a]  = \IK_a[\idx_p, \idx_t]$. 
    The equality between \eqref{eq:third_ses} and \eqref{eq:fourth_ses} holds because $\idx_t < \idx_q$ implies that $\IK_a[\idx_p, \idx_t] \subseteq \IK_a[\idx_p, \idx_{q-1}]$. 
    Therefore, our desired equality \eqref{eq:sufficient_equality_ses} follows.

    This establishes that the sequence is exact. 
    Moreover, the sequence splits because $A^{pq}_{\idx_t}$ is a submodule of a free module, and therefore is free.  
    % This completes the proof.
\end{proof}

\section*{Acknowledgements}
We thank Marzieh Eidi for helpful conversations. We also thank the anonymous referee for helpful comments and suggestions.

%    Bibliographies can be prepared with BibTeX using amsplain,
%    amsalpha, or (for "historical" overviews) natbib style.
\bibliographystyle{amsplain}
\bibliography{references}
%    Insert the bibliography data here.

\end{document}